\documentclass[11pt]{amsart}
\usepackage{amsmath,amsthm}
\usepackage{paralist}
\usepackage{graphics}
\usepackage{epsfig}
\usepackage{graphicx}
\usepackage{epstopdf}
\usepackage[colorlinks=true]{hyperref}
\hypersetup{urlcolor=blue, citecolor=red}

\usepackage{enumerate}
\usepackage{comment}
\usepackage{xcolor}
\usepackage{graphicx}
\usepackage{subfigure}
\usepackage{bbm}
\usepackage{bm}

\usepackage{soul}
\newcommand{\replace}[2]{{\color{black}#2}}
\newcommand{\note}[1]{}

\newcommand{\curlyE}{\mathcal{E}}

\newcommand{\mF}{\mathcal{F}}
\newcommand{\mP}{{\mathcal{P}}}

\newcommand{\mpt}{{\mathcal{P}_2}}
\newcommand{\mptr}{{\mathcal{P}_2(\R)}}
\newcommand{\mpta}{{\mathcal{P}_2^a}}
\newcommand{\mptra}{{\mathcal{P}_2^a(\R)}}

\newcommand{\mW}{\mathcal{W}}
\newcommand{\rhot}{\rho_{\tau}}
\newcommand{\etat}{\eta_{\tau}}
\newcommand{\rhotn}{\rho_{\tau}^{n}}
\newcommand{\etatn}{\eta_{\tau}^{n}}
\newcommand{\rhotnn}{\rho_{\tau}^{n+1}}
\newcommand{\etatnn}{\eta_{\tau}^{n+1}}
\newcommand{\gammat}{\gamma_\tau}
\newcommand{\gammatn}{\gamma_\tau^n}
\newcommand{\gammatnn}{\gamma_\tau^{n+1}}

\newcommand{\R}{\mathbb{R}}
\newcommand{\Rd}{{\mathbb{R}^{d}}}

\newcommand{\argmin}{\mathrm{argmin}}

\newcommand{\sign}{\mathrm{sign}}

\newcommand{\N}{\mathbb{N}}
\newcommand{\ddt}{\frac{\mathrm{d}}{\mathrm{d} t}}
\renewcommand{\d}{\,\mathrm{d}}

\renewcommand{\sign}{\mathrm{sign}}

\newcommand{\A}{\mathcal{A}}
\newcommand{\E}{\mathcal{E}}
\newcommand{\G}{\mathcal{G}}

\newcommand{\sau}{\bm{S}_{\A}^1}
\newcommand{\sad}{\bm{S}_{\A}^2}

\newcommand{\seut}{\bm{S}_{\E}^{1,t}}
\newcommand{\sedt}{\bm{S}_{\E}^{2,t}}

\newcommand{\sgut}{\bm{S}_{\G}^{1,t}}
\newcommand{\sgdt}{\bm{S}_{\G}^{2,t}}

\newcommand{\saut}{\bm{S}_{\A}^{1,t}}
\newcommand{\sadt}{\bm{S}_{\A}^{2,t}}

\newtheorem{remark}{Remark}

\renewcommand{\div}[0]{{\rm{div}}}
\ifpdf
  \DeclareGraphicsExtensions{.eps,.pdf,.png,.jpg}
\else
  \DeclareGraphicsExtensions{.eps}
\fi

\newtheorem{theorem}{Theorem}

\newtheorem{lemma}[theorem]{Lemma}
\newtheorem{proposition}{Proposition}

\newtheorem{definition}[theorem]{Definition}

\title[System of continuity equations with Newtonian interactions]{Measure solutions to a system of continuity equations driven by Newtonian nonlocal interactions}

\author[J.A. Carrillo, M. Di Francesco, A. Esposito, S. Fagioli, and M. Schmidtchen]{}


\email{carrillo@imperial.ac.uk}
\email{marco.difrancesco@univaq.it}
\email{antonio.esposito3@graduate.univaq.it}
\email{simone.fagioli@univaq.it}
\email{m.schmidtchen15@imperial.ac.uk}

\thanks{$^*$ Corresponding author: Jos\'e A. Carrillo}

\begin{document}

\maketitle

\centerline{\scshape Jos\'{e} Antonio Carrillo}
\medskip
{\footnotesize
   \centerline{Department of Mathematics, Imperial College London}
    \centerline{London SW7 2AZ, United Kingdom}
}
\medskip

\centerline{\scshape Marco Di Francesco}
\medskip
{\footnotesize
    \centerline{DISIM - Department of Information Engineering, Computer Science and Mathematics, University of L'Aquila}
    \centerline{Via Vetoio 1 (Coppito) 67100 L'Aquila (AQ) - Italy}
} 

\medskip

\centerline{\scshape Antonio Esposito}
\medskip
{\footnotesize
        \centerline{DISIM - Department of Information Engineering, Computer Science and Mathematics, University of L'Aquila}
    \centerline{Via Vetoio 1 (Coppito) 67100 L'Aquila (AQ) - Italy}
}

\medskip

\centerline{\scshape Simone Fagioli}
\medskip
{\footnotesize
       \centerline{DISIM - Department of Information Engineering, Computer Science and Mathematics, University of L'Aquila}
    \centerline{Via Vetoio 1 (Coppito) 67100 L'Aquila (AQ) - Italy}
}

\medskip

\centerline{\scshape Markus Schmidtchen}
\medskip
{\footnotesize
    \centerline{Department of Mathematics, Imperial College London}
    \centerline{London SW7 2AZ, United Kingdom}
}

\bigskip

\centerline{(Communicated by the associate editor name)}

\begin{abstract}
	We prove global-in-time existence and uniqueness of measure solutions of a nonlocal interaction system of two species in one spatial dimension. For initial data including atomic parts we provide a notion of gradient-flow solutions in terms of the pseudo-inverses of the corresponding cumulative distribution functions, for which the system can be stated as a gradient flow on the Hilbert space $L^2(0,1)^2$ according to the classical theory by Br\'ezis. For absolutely continuous initial data we construct solutions using a minimising movement scheme in the set of probability measures. In addition we show that the scheme preserves finiteness of the $L^m$-norms for all $m\in [1,+\infty]$ and of the second moments. We then provide a characterisation of equilibria and prove that they are achieved (up to time subsequences) in the large time asymptotics. We conclude the paper constructing two examples of non-uniqueness of measure solutions emanating from the same (atomic) initial datum, showing that the notion of gradient flow solution is necessary to single out a unique measure solution.
\end{abstract}


\section{Introduction}
In this work we consider a particular instance of the following nonlocal interaction system for the evolution of two probability measures $\rho$ and $\eta$ on the whole real line
\begin{equation}\label{eq:system}
\begin{cases}
\dfrac{\partial\rho}{\partial t}=\dfrac{\partial}{\partial x} \big(\rho H_1' \star\rho + \rho K_1'\star\eta\big),\\[1em]
\dfrac{\partial \eta}{\partial t}=\dfrac{\partial}{\partial x} \big(\eta H_2' \star\eta + \eta K_2'\star\rho \big).
\end{cases}
\end{equation}
Here $H_1, H_2$ model the way any two agents of the same species interact with one  another (so-called \emph{self-interaction} potentials, or \emph{intraspecific} interaction potentials). Respectively,  $K_1,K_2$ are called the \emph{cross-interaction} potentials, or \emph{interspecific} interaction potentials, as they describe the interaction between any two agents of opposing species.

This model can be easily understood as a natural extension of the well-known aggregation equation (cf. \cite{BCL,BLR11,CDFFLS,MEK99,TBL06}) to two species. This link was first established in the paper \cite{DFF} as the continuous counterpart of a system of ordinary differential equations. More precisely, for $M,N \in \N$ suppose $(x_i)_{i=1}^N$ and $(y_i)_{i=1}^M$ denote the locations of agents of two different species, each of them with masses $\frac{1}{N}$ and $\frac{1}{M}$ respectively. Then, assuming the velocity of any agent is given as an average of the forces exhibited by all other agents upon that agent, one gets
\begin{align*}
	\dot x_i &= - \frac1N \sum_{j\neq i}H_1'(x_i-x_j) - \frac1M\sum_{j} K_1'(x_i-y_j), \quad i=1,...,N,\\
	\dot y_i &= - \frac1M \sum_{j\neq i}H_2'(y_i-y_j) - \frac1N\sum_{j} K_2' (y_i-x_j), \quad i=1,...,M.
\end{align*}
The choice of the interaction potentials depends on the application or the phenomena of interest. In particular in mathematical biology contexts the potentials are often assumed to be radial, i.e. $W(x) = w_W(|x|)$, for $W\in\{H_i, K_i\,|\, i=1,2\}$, i.e. they only depend on the relative distance between any two agents. An interaction potential is said to be \emph{attractive} if $w_W'(|x|) > 0$ and \emph{repulsive} if  $w_W'(|x|) < 0$. The existence theory developed in \cite{DFF} covers the case of $C^1$-potentials $H_i$, $K_i$, $i=1,2$ (with suitable growth conditions at infinity) and provides a semigroup defined in the space of probability measures with finite second moment equipped with the Wasserstein distance, in the spirit of \cite{AGS}. More specifically, the JKO scheme, \cite{JKO}, can be adopted in the special case $K_1=K_2=K$. Then Eq. \eqref{eq:system} can actually be seen as the gradient flow of the functional
\begin{equation}\label{eq:functional_general}
\mF(\rho,\eta):=\frac{1}{2}\int_\R H_1\star\rho\, \d \rho + \frac{1}{2}\int_\R H_2\star\eta\, \d \eta + \int_\R K\star\eta\, \d \rho.
\end{equation}
In this case a slightly lower regularity needs to be required on the potentials $H_1, H_2, K$ as long as all of them are convex up to a quadratic perturbation. Thus, uniqueness can be proven via the notion of $\lambda$-convexity along geodesics of $\mF$, see \cite{McC97,AGS}.

Common interaction potentials for the one species case include power laws $W(x)=|x|^p/p$, as for instance in the case of granular media models, cf. \cite{BCP97, Tos00}. Another possible choice is a combination of power laws of the form $W(x)=|x|^a/a- |x|^b/b$, for $-d<b<a$ where $d$ is the space dimension. These potentials, featuring short-range repulsion and long-range attraction, are typically chosen in the context of swarming models, cf. \cite{BCLR2,BKSUV,CDM16,CHS17,CHM14,FH,FHK,KSUB}. Other typical choices include characteristic functions of sets or Morse potentials
$$
	W(x) = -c_a \exp(-|x|/l_a) + c_r \exp(-|x|/l_r),
$$
or their regularised versions
$
W_p(x) = -c_a \exp(-|x|^p/l_a) + c_r \exp(-|x|^p/l_r),
$
where $c_a,c_r$ and $l_a,l_r$ denote the interaction strength and radius of the attractive (resp. repulsive) part and $p\geq2$,  cf. \cite{CHM14,CMP13,OCBC06}. These potentials display a decaying interaction strength, \textit{e.g.} accounting for biological limitations of visual, acoustic or olfactory sense. The asymptotic behaviour of solutions to one single equation where the repulsion is modelled by non-linear diffusion and the attraction by non-local forces has also received lots of attention in terms of qualitative properties, stationary states and metastability, see \cite{BFH,CCH1,CCH2,CCH15,EK,CCY}  and the references therein, as well as its two-species counterparts, cf. e.g. \cite{DiFranEspFag, CHS17, CFS18, BDFS}, and references therein.

As set out earlier we shall study a particular instance of the above system where all interactions are modelled by Newtonian potentials. More precisely, by setting $N(x):=|x|$,
we consider repulsive Newtonian intraspecific interactions and attractive Newtonian interspecific interactions, i.e. we will deal with the system
\begin{equation}\label{eq:first}
\begin{cases}
\partial_t\rho=\partial_x(-\rho N'\star\rho + \rho N'\star\eta),\\
\partial_t\eta=\partial_x(-\eta N'\star\eta + \eta N'\star\rho).
\end{cases}
\end{equation}
Following \eqref{eq:functional_general}, there is  a natural functional that can be associated to system \eqref{eq:first}, namely
\begin{equation}
\label{eq:enfunctional}
\mF(\rho,\eta):=-\frac{1}{2}\int_\R N\star\rho\, \d \rho - \frac{1}{2}\int_\R N\star\eta\, \d \eta + \int_\R N\star\eta\, \d \rho.
\end{equation}
We mention at this stage that this choice of the functional does not fit the set of assumptions in \cite{DFF}, in that the (repulsive) intraspecific parts of $\mF$ are not defined through convex potential (up to a quadratic perturbation).

The corresponding equation for one species has been attracting a lot of interest. In \cite{BCL} and \cite{BLR11}, the authors provide an $L^\infty$ and an $L^p$-theory for the aggregation equation $\partial_tu+\div(uv)=0$, $v=-\nabla K\star u$, with initial data in $\mpt(\Rd)\cap L^p(\Rd)$, where $d\ge2$ and $\mpt(\Rd)$ denotes the set of probability measures with bounded second order moments. They consider radially symmetric kernels whose singularity is of order $|x|^\alpha$, $\alpha>2-d$, at the origin. In particular, the authors prove local well-posedness in $\mpt(\Rd)\cap L^p(\Rd)$ for $p>p_s$, where $p_s=p_s(d,\alpha)$. Moreover, when $K(x)=|x|$, the exponent $p_s=\frac{d}{d-1}$ is sharp since for any $p<\frac{d}{d-1}$ the solution instantaneously concentrates mass at the origin for initial data in $\mpt(\Rd)\cap L^p(\Rd)$. Global well-posedness of solutions with initial data in $\mpt(\Rd)$ was proven in \cite{CDFFLS} for a general class of potentials including in particular $K(x)=|x|$. The gradient flow structure was crucial to show a unique continuation after blow-up of solutions to the aggregation equation. Let us also mention \cite{BLL12} that provides a well-posedness theory of compactly supported $L^1\cap L^\infty$-solutions for the Newtonian potentials in $d\geq 2$. The gradient flow structure introduced in \cite{CDFFLS} in the particular case of $K(x)=|x|$ in one dimension was further developed in  \cite{BCDFP}, where the authors prove the equivalence of  the Wasserstein gradient flow for
$$
\partial_t\rho=\partial_x(\rho\partial_xW\star\rho), \qquad x\in\R, \quad t>0,
$$
with $W(x)=-|x|$ or $W(x)=|x|$, and the notion of entropy solution of a scalar nonlinear conservation law of Burgers-type
$$
\partial_tF+\partial_xg(F)=0, \qquad x\in\R, \quad t>0,
$$
where
$$
g(F)=F^2-F \qquad \text{or}\qquad g(F)=-F^2+F,
$$
in the repulsive or attractive case respectively, being $F(t,x)=\int_{-\infty}^{x}\rho(t,x)\d x$. Such a result is relevant in particular in the repulsive case, as it shows that all point particles initial data evolve into an $L^1$-density on $t\in(0,+\infty)$ as a simple consequence of the uniqueness of entropy solutions to the corresponding scalar conservation law, \cite{kru}. More precisely, a point particle $\rho_0=\delta_0$ in the repulsive case corresponds to the initial datum $F_0=\mathbf{1}_{[0,+\infty)}$ for the equation $F_t+(F^2-F)_x = 0$, and the discontinuity can be resolved (in a weak solution sense) either via a stationary Heaviside function or through a \emph{rarefaction wave} with time-decaying slope connecting the two states $0$ and $1$. As the flux $g(F)=F^2-F$ is convex, the latter is the only admissible solution in the entropy sense (see e.g. \cite{Dafermos}). Therefore, the equivalence result in \cite{BCDFP} implies that the distributional derivative $\rho=\partial_x F$ is the only gradient flow solution to the repulsive aggregation equation $\rho_t = -\partial_x (\rho \partial_x(|x|\star \rho))$. Notice that such a solution satisfies $\rho(t,\cdot)\in L^\infty(\R)$ for all $t>0$, whereas the initial condition $\rho_0$ is an atomic measure.

The occurrence of such a \emph{smoothing} effect in the one-species repulsive case suggests that similar phenomena may be observed in the two-species case, at least in one space dimension. Understanding such an issue is one of the purposes of this work. However, the equivalence to the $2\times 2$ system of conservation laws
\begin{equation}\label{eq:hypsys_intro}
\begin{cases}
	\partial_t F +2(F-G) \partial_x F=0,\\
    \partial_t G +2(G-F) \partial_x G=0,
\end{cases}
\qquad F(t,x)=\int_{-\infty}^x \rho(t,y) \d y,\qquad G(t,x)=\int_{-\infty}^x \eta(t,y) \d y,
\end{equation}
does not provide any useful insights in this case, as we shall discuss  in detail in Section \ref{sec:atomic}. We would like to stress at this stage that the persistence of an atomic part for one of the two species in \eqref{eq:first} would make the definition of measure solutions rather difficult, as the velocity fields are given by convolutions of the solution with a discontinuous function. In the $(F,G)$ version \eqref{eq:hypsys_intro} this corresponds to the impossibility of e.g. multiplying a discontinuous function $F-G$ by an atomic measure $\partial_x F$. On the other hand, we shall see that the equivalence to $L^2(0,1)^2$-gradient flows in the pseudo-inverse formalism (see e.g. \cite{Brenier,CT,BCDFP}) provides a natural way to state a suitable notion of solution for \eqref{eq:first} with measure initial data giving rise to a \emph{unique} solution for all times. As for the $L^p$-regularity of solutions, we mention here that a system similar to \eqref{eq:first} with the addition of linear diffusion in both components was studied in the context of semiconductor device modelling, see e.g. \cite{arnold} in which solutions are shown to maintain the finiteness of the $L^p$-norms for $p\in (1,+\infty)$. As we will show in our paper, the same holds in the one-dimensional diffusion free case \eqref{eq:first}. However, when initial data feature an atomic part, the attractive part in the functional $\mF$ may inhibit solutions to instantaneously become $L^1$-densities. This may happen for instance when the two species share an atomic part at the same position initially, see Section \ref{sec:atomic} below.

Another interesting issue related to \eqref{eq:first} is its asymptotic behaviour for large times. The one species case features total collapse, i.e. the formation of one point particle in a finite time in the attractive case with all the mass of the system, see \cite{CDFFLS}, and large time decay to zero in the repulsive case (as a consequence of the results in \cite{BCDFP} and on classical results on the large time behaviour for scalar conservation laws, see e.g. \cite{liu}). The asymptotic behaviour in the $2\times 2$ case \eqref{eq:first} is much more complex. Finite time concentration for smooth initial data cannot happen in the view of the non-expansiveness of the $L^p$-norms  proven in the present work. Similarly the case for the large time decay to zero is impossible, as we will construct explicit solutions featuring a steady atomic part for all times. We shall prove that the $\omega$-limit in a suitable topology for \eqref{eq:first} is a subset of $\{\rho=\eta\}$, which also coincides with the minimising set of the corresponding functional $\mF$.
The rest of this paper is organised as follows:
\begin{itemize}
  \item Section \ref{sec:preliminaries} contains preliminary concepts on gradient flows in  Wasserstein spaces and about the one-dimensional case in particular.
  \item Section \ref{sec:existence} deals with the existence and uniqueness of solutions. We first prove it in Subsection \ref{subsec:measure}, for the notion of solution introduced in Definition \ref{def:solutionmeasures}. In Subsection \ref{subsec:densities_case} we consider the case of densities as initial conditions, more precisely in $L^m(\R)$ for some $m\in (1,+\infty]$, and we show that a suitable notion of gradient flow solution in the Wasserstein sense (see Definition \ref{def:gradflow}) can be achieved via the Jordan-Kinderlehrer-Otto scheme, which also allows to prove that the $L^m$-regularity is maintained. In addition a uniform-in-time control of the second moment is obtained. Moreover, we prove that our solutions also satisfy Definition \ref{def:solutionmeasures} given the additional regularity. All the results on the absolutely continuous case are collected in Theorem \ref{thm:regularity}.
  \item Section \ref{sec:steady_states} contains a detailed study of the steady states for \eqref{eq:first}, as well as of the minimisers of \eqref{eq:enfunctional}. A characterisation of the steady states is provided in Proposition \ref{prop:char_steady}. A consequent result concerning the asymptotic behaviour is provided in Theorem \ref{thm:omega}.
  \item Section \ref{sec:atomic} describes two relevant examples of atomic initial data. In both cases, non uniqueness of weak measure solutions is shown, and the relevant gradient flow solution is singled out as well. These two examples allow to conclude interesting properties related with the occurrence or not of the smoothing effect (or lack thereof) of initial atomic parts, see Remarks \ref{rem:example1} and \ref{rem:example2}. The link with the hyperbolic system \eqref{eq:hypsys_intro} is described in detail in Subsection \ref{subsec:hyp} leading to several open problems.
\end{itemize}



\section{Preliminaries}\label{sec:preliminaries} This section is devoted to setting up the framework to show existence and uniqueness of  solutions to the system
\begin{subequations}
\label{eq:full_system}
\begin{equation}\label{eq:full_system_no_data}
\begin{cases}
\partial_t\rho=-\partial_x(\rho N'\star\rho)+\partial_x(\rho N'\star\eta), \\
\partial_t\eta=-\partial_x(\eta N'\star\eta)+\partial_x(\eta N'\star\rho),
\end{cases}
\end{equation}
with Newtonian interactions, $N(x)=|x|$.
Throughout the paper, $\rho=\rho(t)$ and $\eta=\eta(t)$ will be considered as time dependent curves with values on the set $\mP(\R)$ of probability measures on $\R$. System \eqref{eq:full_system_no_data} is equipped with initial data
\begin{align}
	\rho(0) = \rho_0, \qquad \mbox{and} \qquad \eta(0)=\eta_0,
\end{align}
\end{subequations}
for some $\rho_0,\eta_0\in \mathcal{P}(\R)$. Moreover, we write $\mptr$ to denote the set of  probability measures with finite second moment, i.e.
$$
\mptr=\left\{\mu\in\mP(\R) \, |\, m_2(\mu)<+\infty\right\},
\mbox{ where } m_2(\mu)=\int_{\R}|x|^2\, \d \mu(x).$$
In the following we shall use the symbol $\mptra$ referring to elements of $\mptr$ which are absolutely continuous with respect to the Lebesgue measure.
Consider a measure $\mu\in\mP(\R)$ and a Borel map $T:\R\to\R$. We denote by $\nu = T_{\#}\mu$ the push-forward of $\mu$ through $T$, defined by
$$
\nu(A)=\mu(T^{-1}(A))
$$
for all Borel sets $A\subset\R$. We refer to $T$ as the transport map pushing $\mu$ to $\nu$. Next let us equip the set $\mptr$ with the $2$-Wasserstein distance. For any measures $\mu,\nu\in \mptr$ it is defined as
\begin{equation}\label{wass}
W_2(\mu,\nu)=\left(\inf_{\gamma\in\Gamma(\mu,\nu)}\int_{\R^2}|x-y|^2\, \d  \gamma(x,y)\right)^{1/2},
\end{equation}
where $\Gamma(\mu,\nu)$ is the class of transport plans between $\mu$ and $\nu$, that is,
\begin{align*}
	\Gamma(\mu, \nu):= \{ \gamma\in \mP(\R^2)\,|\, \pi^1_{\#}\gamma = \mu, \,\pi^2_{\#}\gamma = \nu\},
\end{align*}
where $\pi^i:\R\times\R\rightarrow\R$, $i=1,2$, denotes the projection operator on the $i^\mathrm{th}$ component of the product space $\R^2$.  Setting $\Gamma_0(\mu,\nu)$ as the class of optimal plans, i.e. minimisers of \eqref{wass}, the Wasserstein distance becomes
$$
	W_2^2(\mu,\nu)=\int_{\R^2}|x-y|^2\, \d \gamma(x,y),
$$
for any $\gamma\in\Gamma_0(\mu,\nu)$. The set $\mptr$ equipped with the $2$-Wasserstein metric is a complete metric space which can be seen as a length space, see for instance \cite{AGS,S,V1,V2}.

\begin{remark}
\label{rem:inequalitymom}
 Given two measures $\mu,\nu\in\mptr$, by using  the elementary inequality $|y|^2\le2|x|^2+2|x-y|^2$ and the above definition of the 2-Wasserstein distance, one can easily show the inequality
$$
m_2(\nu)\le2m_2(\mu)+2W_2^2(\mu,\nu),
$$
which will be used later on.
\end{remark}

Next we introduce the notion of the Fr\'echet sub-differential in the space of probability measures.
\begin{definition}[Fr\'{e}chet sub-differential in $\mptr$]\label{def:subdiff}
Let $\phi:\mptr\to(-\infty,+\infty]$ be a proper and lower semicontinuous functional, and let $\mu\in D(\phi):=\{\mu\in \mptr\,|\, \phi(\mu)<\infty\}$. We say that $v\in L^2(\R; \mu)$ belongs to the \emph{Fr\'{e}chet sub-differential} at $\mu$, denoted by $\partial\phi(\mu)$, if
$$
\phi(\nu)-\phi(\mu)\ge\inf_{\gamma\in\Gamma_0(\mu,\nu)}\int_{\R\times\R}v(x)(y-x)\, \d \gamma(x,y)+o(W_2(\mu,\nu)).
$$
Moreover, if $\partial\phi(\mu)\neq\emptyset$ we denote by $\partial^0\phi(\mu)$ the element of minimal $L^2(\R; \mu)$-norm in $\partial\phi(\mu)$.
\end{definition}
This definition will play a crucial role when introducing the notion of gradient flow solutions to system \eqref{eq:full_system} later on, cf. Section \ref{subsec:densities_case}.

A curve $\mu:[0,1]\to\mptr$ is a \emph{constant speed geodesic} if $W_2(\mu(s),\mu(t))=(t-s)W_2(\mu(0),\mu(1))$ for all $0\le s\le t\le1$. Due to \cite[Theorem 7.2.2]{AGS}, a constant speed geodesic connecting $\mu$ and $\nu$ can be written as
$$
	\gamma_t=\left((1-t)\pi^1+t\pi^2\right)_{\#}\gamma,
$$
where $\gamma\in\Gamma_0(\mu,\nu)$ and thus $\mu=\gamma_0$ and $\nu=\gamma_1$. In the literature $\gamma_t$ is also known as \emph{McCann interpolation}, cf. \cite{McC97}. Next, we introduce a modified notion of convexity, the so-called $\lambda$-geodesic convexity, which is of paramount importance in the study of gradient flows in the  metric space $\mptr$.
\begin{definition}[$\lambda$-geodesic convexity]
Let $\lambda\in \R$. A functional $\phi:\mptr\to(-\infty,+\infty]$ is said to be \emph{$\lambda$-geodesically convex} in $\mptr$ if for every  $\mu,\nu\in\mptr$ there exists $\gamma\in\Gamma_0(\mu,\nu)$ such that
$$
\phi(\gamma_t)\le(1-t)\phi(\mu)+t\phi(\nu)-\frac{\lambda}{2}t(1-t)W_2^2(\mu,\nu),
$$
for any  $t\in[0,1]$.
\end{definition}
It is necessary to recall that the $\lambda$-geodesic convexity is strictly linked with the concept of \textit{k-flow}.

\begin{definition}[$k$-flow]
A semigroup $S_{\phi}:[0,+\infty]\times\mptr\to\mptr$ is a $k$-flow for a functional $\phi:\mptr \to (-\infty,\infty]$ with respect to $W_2$ if, for an arbitrary $\mu\in\mptr$, the curve $t\mapsto S_\phi^t\mu$ is absolutely continuous on $[0,+\infty[$ and satisfies the evolution variational inequality (\textbf{E.V.I.})
\begin{equation}
\frac{1}{2}\frac{\d^+}{\d t}W_2^2(S_\phi^t\mu,\tilde{\mu})+\frac{k}{2}W_2^2(S_\phi^t\mu,\tilde{\mu}) \le \phi(\tilde{\mu})-\phi(S_\phi^t\mu)
\end{equation}
for all $t>0$, with respect to any reference measure $\tilde{\mu}\in\mptr$ such that $\phi(\tilde{\mu})<\infty$.
\end{definition}
As already mentioned, the previous concepts of $k$-flow and $\lambda$-convexity are closely intertwined. Indeed, a $\lambda-$convex functional possesses a uniquely determined $k-$flow for $k\geq \lambda$, and if a functional possesses a $k-$flow, then it is $\lambda-$convex with $\lambda \geq k$, cf. Refs.  \cite{AGS,DFM,MMCS}, for further details. The notion of k-flow will be of great help in the use of the flow interchange technique, cf. Subsection \ref{subsec:densities_case}.

Now, let $\mu_t\in AC([0,+\infty);\mptr)$ be an absolutely continuous curve in $\mptr$. We can define the metric derivative of $\mu_t$  as
$$
|\mu_t'|(t):=\limsup_{h\to0}\frac{W_2(\mu_{t+h},\mu_t)}{|h|},
$$
which is well-defined almost everywhere since $\mu_t$ is an absolutely continuous curve, cf. \cite{AGS,S}.

As system \eqref{eq:full_system} describes the evolution of two interacting species, it is necessary to work on the product space $\mptr\times\mptr$ equipped with the 2-Wasserstein product distance, which is defined in the canonical way
$$
\mW_2^2(\gamma,\tilde \gamma)=W_2^2(\rho,\tilde \rho)+W_2^2(\eta,\tilde \eta),
$$
for all $\gamma=(\rho,\eta), \tilde \gamma=(\tilde \rho,\tilde \eta)$ belonging to $\mptr\times\mptr$. Now, let us introduce another crucial tool in our paper. For a given $\mu\in\mptr$ its cumulative distribution function is given by
\begin{equation}\label{eq:distribution-functions}
	F_\mu(x)=\mu((-\infty,x]).
\end{equation}
Since $F_\mu$ is a non-decreasing, right continuous function such that
\begin{align*}
	\lim_{x\rightarrow -\infty} F_\mu(x) = 0, \quad \mbox{and} \quad 	\lim_{x\rightarrow +\infty} F_\mu(x) = 1,
\end{align*}
 we may define the pseudo-inverse function $X_\mu$ associated to $F_\mu$, by
\begin{equation}
	\label{eq:pseudoinverse_def}
	X_\mu(s):=\inf_{x\in \R}\{F_\mu(x)>s\},
\end{equation}
for any $s\in (0,1)$. It is easy to see that $X_\mu$ is right-continuous and non-decreasing as well. Having introduced the pseudo-inverse, let us now recall some of its important properties. First we notice that it is possible to pass from $X_\mu$ to $F_\mu$ as follows
\begin{equation}
	\label{eq:F_intermsof_X}
	F_\mu(x)=\int_0^1\mathbf{1}_{(-\infty,x]}(X_\mu(s))\, \d s=|\{X_\mu(s)\le x\}|.
\end{equation}

For any probability measure $\mu\in\mptr$ and the pseudo-inverse , $X_\mu$, associated to it, we have
\begin{equation}
	\label{eq:CoV_PseudoInverse}
	\int_\R f(x)\, \d \mu(x)=\int_0^1f(X_\mu(s))\, \d s,
\end{equation}
for every bounded continuous function $f$. Moreover, for $\mu,\nu\in\mptr$, the Hoeffding-Fr\'echet theorem \cite[Section 3.1]{RR} allows us to represent the 2-Wasserstein distance, $W_2(\mu,\nu)$, in terms of the associated pseudo-inverse functions as
\begin{equation}
	\label{eq:W2PseudoInverses}
	W_2^2(\mu,\nu)=\int_0^1|X_\mu(s)-X_\nu(s)|^2\, \d s,
\end{equation}
since the optimal plan is given by $(X_\mu(s)\otimes X_\nu(s))_{\#}\mathcal{L}$, where $\mathcal{L}$ is the Lebesgue measure on the interval $[0,1]$, cf. also \cite{V1,CT}. We have seen that for every $\mu\in\mptr$ we can construct a non-decreasing $X_\mu$ according to \eqref{eq:pseudoinverse_def}, and by the change of variables formula \eqref{eq:CoV_PseudoInverse} we also know that $X_\mu$ is square integrable. We now recall that this mapping is indeed a distance-preserving bijection between the space of probability measure with finite second moments and the convex cone of non-decreasing $L^2$-functions
\begin{equation}\label{eq:cone}
	\mathcal{C}:=\{f\in L^2(0,1)\,|\,f\ \mbox{is non-decreasing}\} \subset L^2(0,1).
\end{equation}

\begin{proposition}[$\mu\mapsto X_\mu$  is an isometry]
	The map
	\begin{equation}\label{eq:isometry}
		\Psi:\mptr\ni \mu \mapsto X_\mu\in\mathcal{C},
	\end{equation}
	mapping probability measures onto the convex cone of non-decreasing $L^2$-functions is an isometry.
\end{proposition}

Let us introduce the notion of sub-differential for functions in $L^2(0,1)^2$.
\begin{definition}[Fr\'echet sub-differential in $L^2(0,1)^2$]
For a given proper and lower semi-continuous functional $\mathfrak{F}$ on $L^2(0,1)^2$ we say that $Z = (X, Y)\in L^2(0,1)^2$ belongs to the sub-differential of $\mathfrak{F}$ at $\tilde Z=(\tilde X, \tilde Y)\in L^2(0,1)^2$ if and only if
\begin{equation*}
\mathfrak{F}(R)-\mathfrak{F}(\tilde Z)\ge\int_0^1 \left[X (s)(R_1(s)-\tilde X(s))+Y (s)(R_2(s)- \tilde Y(s))\right]\, \d s+o(\|R-\tilde Z\|),
\end{equation*}
as $\|R-\tilde Z\|\to0$, with the notation $R=(R_1,R_2)\in L^2(0,1)^2$. The sub-differential of $\mathfrak{F}$ at $\tilde Z$ is denoted by $\partial \mathfrak{F}(\tilde Z)$, and if $\partial \mathfrak{F}(\tilde Z)\neq \emptyset$ then we denote by $\partial^0 \mathfrak{F}(\tilde Z)$ the element of minimal $L^2$-norm of $\partial \mathfrak{F}(\tilde Z)$.
\end{definition}

\begin{remark}[Mass normalisation]
  In the most general situation possible, the two species, $\rho$ and $\eta$, have different masses $M_\rho, M_\eta>0$. The change of variables
  \[\widetilde{\rho}=\frac{1}{M_\rho}\rho,\qquad \text{and} \qquad\widetilde{\eta}=\frac{1}{M_\eta}\eta,\]
  allows to rewrite system \eqref{eq:full_system_no_data} (by dropping the tildes) as
\begin{equation*}
\begin{cases}
\partial_t\rho=-M_\rho \partial_x(\rho N'\star\rho)+M_\eta\partial_x(\rho N'\star\eta), \\
\partial_t\eta=-M_\eta\partial_x(\eta N'\star\eta)+M_\rho\partial_x(\eta N'\star\rho),
\end{cases}
\end{equation*}
and the gradient flow structure in the product Wasserstein metric $\mW_2$ would be lost because the two interspecific potentials are different. However, this problem can be overcome by using a weighted version of the $\mW_2$ product distance of the form
\[\mW_2^2((\rho,\eta),(\widetilde{\rho},\widetilde{\eta}))=W_2^2(\rho,\widetilde{\rho})+\frac{M_\eta}{M_\rho}W_2^2(\eta,\widetilde{\eta}),\]
as done in \cite{DFF}. As these multiplying constants $M_\rho, M_\eta$ do not bring significant technical difficulties (while making the notation much heavier), for the sake of convenience we shall assume throughout the whole paper that $M_\rho=M_\eta=1$, unless specified otherwise.
\end{remark}

\section{Existence and Uniqueness}\label{sec:existence}
In this section we provide the mathematical theory for system \eqref{eq:full_system}. In the first subsection we will deal with the case of general probability measures in $\mptr\times\mptr$ as initial conditions. In the second subsection we shall restrict ourselves to the case of measures that are absolutely continuous with respect to the Lebesgue measure. In the former we will provide a notion of solutions that is linked to the concept of gradient flows in Hilbert space \emph{a-la} Br\'ezis \cite{Brezis}, working with the pseudo-inverse formulation of the problem. In the latter, a better regularity can be achieved and the theory is developed in the framework of gradient flows in Wasserstein space, \cite{AGS}.
Before entering the details, let us recall the definition of \textit{interaction energy functional} $\mF$ in \eqref{eq:enfunctional}: for all $(\rho,\eta)\in\mptr\times\mptr$ we set
\begin{equation*}
\mF(\rho,\eta):=-\frac{1}{2}\int_\R N\star\rho\, \d \rho - \frac{1}{2}\int_\R N\star\eta\, \d \eta + \int_\R N\star\eta\, \d \rho,
\end{equation*}
which is well defined due to the control on the second order moment.
\subsection{General Measures Initial Data}\label{subsec:measure}
In this first subsection we will use the concept of $L^2$-gradient flow by studying  system \eqref{eq:system} in terms of the pseudo-inverse functions $X_\rho$ and $X_\eta$ defined in Section \ref{sec:preliminaries}.
Throughout the rest of this section we set $X:=X_\rho$ and $Y:=X_\eta$ to simplify the notation. Hence, system \eqref{eq:full_system} (formally) becomes
\begin{equation}\label{eq:system-pseudo-inverse}
\begin{cases}
	\dfrac{\partial X}{\partial t} = \displaystyle\int_0^1 \sign(X(z)- X(\xi))\d \xi - \int_0^1 \sign(X(z) - Y(\xi))\d \xi,\\[0.4cm]
    \dfrac{\partial Y}{\partial t} = \displaystyle\int_0^1 \sign(Y(z) - Y(\xi))\d \xi - \int_0^1 \sign(Y(z) - X(\xi))\d \xi,
\end{cases}
\end{equation}
for $s\in(0,1)$ and $t\ge0$, cf. \cite{B,CT,lt} for similar computations. In order to give a meaning to the above system in the case of $\mu$ or $\eta$ having atoms, we use the convention $\sign(0)=0$. In terms of the pseudo inverses $X$ and $Y$, the functional $\mF(\rho,\eta)$ becomes
\begin{equation}
\label{eq:functional_pseudoinverse}
\begin{split}
\mF(\rho,\eta)=\mathfrak{F}(X,Y)=&-\frac{1}{2}\int_0^1\int_0^1|X(z)-X(\xi)|\, \d z\, \d \xi-\frac{1}{2}\int_0^1\int_0^1|Y(z)-Y(\xi)|\, \d z\, \d \xi\\&+\int_0^1\int_0^1|X(z)-Y(\xi)|\, \d z\, \d \xi.
\end{split}
\end{equation}
In the remainder of this section we shall see that \eqref{eq:system-pseudo-inverse} is the $L^2\times L^2$-gradient flow associated to (an extended version of) the energy functional \eqref{eq:functional_pseudoinverse}.

\begin{remark}\label{rem:Fsumof3}
Later on in the paper we need to distinguish between the self-interaction part of $\mathfrak{F}$ and its cross-interaction part. Thus, let us rewrite $\mathfrak{F}$ as
$$
	\mathfrak{F}(X,Y) = S(X)+S(Y)+K(X,Y),
$$
where $S$ is the energy functional arising from the self-interactions and $K$ is associated to the cross-interaction.
\end{remark}
Following the procedure of \cite{BCDFP}, since we are dealing with distribution of particles, we have to ensure that the flow remains in the set $\mathcal{C}\times\mathcal{C}$, with $\mathcal{C}$ defined in \eqref{eq:cone}, see also \cite{Boll-br-loe,Brenier}. Hence, for $X\in L^2(0,1)$ given, we introduce the \textit{indicator function} of $\mathcal{C}$, defined by,
\begin{equation}
\mathcal{I}_{\mathcal{C}}(X)=
\begin{cases}
0, &\text{if}\ X\in\mathcal{C},\\
+\infty, &\text{otherwise}.
\end{cases}
\end{equation}
Thus we consider the extended functional
\begin{equation}
\bar{\mathfrak{F}}(X,Y)=\mathfrak{F}(X,Y)+\mathcal{I}_C(X)+\mathcal{I}_C(Y).
\end{equation}
In \cite[Proposition 2.8]{BCDFP} the authors proved that the self-interaction part of $\mathfrak{F}$, i.e. $S$, is actually linear when restricted to $\mathcal{C}$. Let us recall this result in the next proposition.
\begin{proposition}
\label{prop:selflinear}
Let $X\in\mathcal{C}$. Then
$$
S(X)=\int_0^1(1-2z)X(z)\, \d z.
$$
\end{proposition}
As a trivial consequence of  Proposition \ref{prop:selflinear} we have the following result.
\begin{proposition}\label{prop:convexityforpseudo}
The functional $\tilde{\mathfrak{F}}$ is convex on $L^2(0,1)^2$.
\end{proposition}
\begin{proof}
The proof is trivial since $K$, the cross-interaction part  of $\bar{\mathfrak{F}}$, is convex due to the convexity of the Newtonian potential $N$. Moreover, from \cite[Proposition 2.9]{BCDFP} we argue that the remaining part is convex.
\end{proof}
We now present the definition of $L^2$-gradient flow solutions to  system \eqref{eq:system-pseudo-inverse}.
\begin{definition}\label{def:l2gradflow}
Let $(X_0,Y_0)\in \mathcal{C}\times \mathcal{C}$. An absolutely continuous curve $(X(t,\cdot),Y(t,\cdot))\in L^2(0,1)^2$, $t\ge 0$, is a gradient flow for the functional $\bar{\mathfrak{F}}$ if $Z(t):=(X(t),Y(t))$ is a Lipschitz function on $[0,+\infty)$, i.e., \replace{$\frac{dZ}{dt}\in L^\infty(0,+\infty;\mathcal{C}\times\mathcal{C})$}{$\frac{dZ}{dt}\in L^\infty(0,+\infty;L^2(0,1)^2)$} (in the sense of distributions) and if
it satisfies the sub-differential inclusion
\item
	\begin{equation}\label{eq:l2gradflow}
	\ddt \begin{pmatrix}
X(t,\cdot)\\
Y(t,\cdot)
	\end{pmatrix}\in-\partial\bar{\mathfrak{F}}\big[\big(X(t,\cdot),Y(t,\cdot)\big)\big]
	\end{equation}
for every $t>0$ with $(X(0,\cdot),Y(0,\cdot))=(X_0(\cdot),Y_0(\cdot))$.
\end{definition}
We observe that the assumption $(X_0,Y_0)\in \mathcal{C}\times \mathcal{C}$ is natural as $X_0$ (respectively $Y_0$) is the pseudo-inverse of the cumulative distribution of the initial measure $\rho_0$ (respectively $\eta_0$). We also observe that this assumption easily implies $\partial\bar{\mathfrak{F}}\big[\big(X_0,Y_0\big)\big]\neq\emptyset$.
\begin{remark}
The gradient flow  notion defined in Definition \ref{def:l2gradflow} is taken from the book \cite[Theorem 3.1]{Brezis}. Actually, in \cite[Theorem 3.1]{Brezis} the following extra condition is required, namely
{\color{black}
\begin{equation}
    \label{eq:l2gradflow_timederivative}
      \left\|\frac{\d Z}{\d t}\right\|_{L^\infty((0,+\infty);L^2(0,1)^2)}\leq\left\|\partial^0\bar{\mathfrak{F}}\big[\big(X_0,Y_0\big)\big]\right\|_{L^2(0,1)^2}.
\end{equation}
}
According to \cite[Theorem 3.1]{Brezis} a solution in the sense of Definition \ref{def:l2gradflow} in conjunction with  \eqref{eq:l2gradflow_timederivative}  directly verifies the following properties:
\begin{enumerate}
\item $Z$ admits a right derivative for every $t\in[0,+\infty)$ and
$$
\frac{d^+Z}{dt}(t)=-\partial^0\bar{\mathfrak{F}}\left[Z(t)\right],
$$
for every $t\in[0,+\infty)$;
\item the function $t\mapsto\partial^0\bar{\mathfrak{F}}\left[Z(t)\right]$ is right continuous and the function $t\mapsto\left\|\partial^0\bar{\mathfrak{F}}\left[Z(t)\right]\right\|_{L^2(0,1)^2}$ is non-increasing;
\item if $Z_{1,t}:=(X_1(t,\cdot),Y_1(t,\cdot))$ and $Z_{2,t}:=(X_2(t,\cdot),Y_2(t,\cdot))$ are two solutions to system \eqref{eq:l2gradflow}, then there holds
\begin{equation*}
\|Z_{1,t}-Z_{2,t}\|_{L^2\times L^2}\le\|Z_{1,0}-Z_{2,0}\|_{L^2\times L^2},
\end{equation*}
for all $t\ge0$.
\end{enumerate}
On the other hand, \cite[Section 9.6, Theorem 3]{evans} shows that condition \eqref{eq:l2gradflow_timederivative} can be avoided in order to prove uniqueness. In fact, the estimate \eqref{eq:l2gradflow_timederivative} can be proven as a consequence of the properties stated in Definition \ref{def:l2gradflow}.
\end{remark}
Since we know that $\bar{\mathfrak{F}}$ is a proper, lower semi-continuous, and convex functional on the Hilbert space $L^2(0,1)^2$, it easy to show that $\partial\bar{\mathfrak{F}}$ is a maximal monotone operator. Thus we can apply the theory of Br\'{e}zis \cite[Theorem 3.1]{Brezis} combined with \cite[Section 9.6, Theorem 3]{evans} in order to prove existence and uniqueness of an absolutely continuous curve satisfying the differential inclusion above.
\begin{theorem}\label{thm:l2gradientflow}
Let $(X_0,Y_0)\in\mathcal{C}\times\mathcal{C}$. There exists a unique solution $(X(t,\cdot),Y(t,\cdot))$ in the sense of Definition \ref{def:l2gradflow} with initial datum $(X_0,Y_0)$.
\end{theorem}
Now, let us go back to system \eqref{eq:full_system} and state our definition of solution.
\begin{definition}\label{def:solutionmeasures}
Let $\gamma_0=(\rho_0,\eta_0)\in\mptr\times\mptr$. An absolutely continuous curve $\gamma(t)=(\rho(t),\eta(t)):[0,T]\to\mptr\times\mptr$ is a \emph{gradient flow} solution to  system \eqref{eq:full_system} if the  pseudo-inverses $(X(t,\cdot),Y(t,\cdot))\in\mathcal{C}\times\mathcal{C}$ of the space cumulative distribution functions associated to  $(\rho(t,\cdot),\eta(t,\cdot))$ are a solution to system \eqref{eq:system-pseudo-inverse} in the sense of Definition \ref{def:l2gradflow} with initial datum $(X_0,Y_0)=(X_{\rho_0},Y_{\eta_0})$.
\end{definition}

According to Definition \ref{def:solutionmeasures} the following theorem is then a consequence of the isometry \eqref{eq:isometry} and Theorem \ref{thm:l2gradientflow}.
\begin{theorem}\label{thm:solutionmeasures}
Let $\gamma_0=(\rho_0,\eta_0)\in\mptr\times\mptr$. There exists  a unique solution to the system \eqref{eq:full_system} in the sense of Definition \ref{def:solutionmeasures}.
\end{theorem}

So far we assumed that the link between \eqref{eq:full_system} and \eqref{eq:system-pseudo-inverse} is somewhat natural and we just referred to similar situations in the literature. However, the theory developed in this subsection would be somewhat meaningless if we did not show that the concept of solution in Definition \ref{def:solutionmeasures} extends a more classical notion of solution for \eqref{eq:full_system}. The following subsection is dedicated to establishing exactly this link.

\subsection{Absolutely Continuous Initial Data}\label{subsec:densities_case}
In this subsection we  consider the case of densities as initial data. Following the approach of \cite{AGS} combined with the results from \cite{B,BCDFP,CDFFLS,CFP12,DFF}, we pose  system \eqref{eq:full_system} as the gradient flow of the interaction energy functional \eqref{eq:enfunctional} (that we recall here for the reader's convenience)
\begin{equation*}
\mF(\rho,\eta)=-\frac{1}{2}\int_\R N\star\rho\, \d \rho - \frac{1}{2}\int_\R N\star\eta\, \d \eta + \int_\R N\star\eta\, \d \rho,
\end{equation*}
for all $(\rho,\eta)\in\mptra\times\mptra$, and $N(x)=|x|$, for all $x\in\R$.
\begin{definition}\label{def:gradflow}
Given any $\gamma_0=(\rho_0,\eta_0)\in\mptra\times\mptra$, an absolutely continuous curve $\gamma(t)=(\rho(t),\eta(t)):[0,T]\to\mptra\times\mptra$ is a gradient flow for $\mF$ if $\rho(t)$ and $\eta(t)$ solve the following system in the distributional sense
\begin{equation}\label{eq:gradflowsys}
\begin{cases}
\partial_t\rho(t)+\partial_x(\rho(t)v_{1}(t))=0, \\
\partial_t\eta(t)+\partial_x(\eta(t)v_{2}(t))=0,
\end{cases}
\end{equation}
with initial datum $\gamma_0$ and the velocity field $v(t)=(v_{1}(t),v_{2}(t))$ such that
$$
	v_{i}(t)=- \,(\partial^0\mF[\gamma(t)])_i
$$
for $i=1,2$ and
\begin{equation*}
\|v(t)\|_{L^2(\gamma(t))}=|\gamma'|(t),
\end{equation*}
for a.e. $t>0$.
\end{definition}
Note that it is easy to check that the element of minimal norm in $\partial N(x)$ is given by
$$
\partial^0N(x)=
\begin{cases}
\sign(x), & x\neq 0,\\
0,       & x=0.
\end{cases}
$$
Using the results obtained in \cite{BCDFP,CDFFLS,CFP12,DFF}, we easily get the following proposition.
\begin{proposition}\label{prop:geodconvandsubdiff}
The functional $\mF$ is $\lambda$-geodesically convex on $\mptra\times \mptra$ for all $\lambda\le0$. Moreover, for all $\rho,\eta\in \mptra$ the vector field
\begin{equation}\label{eq:subdiff_sign}
\partial^0\mF[\rho,\eta]=\begin{pmatrix}
-\partial^0 N\star\rho+\partial^0N\star\eta\\
-\partial^0 N\star\eta+\partial^0N\star\rho
\end{pmatrix},
\end{equation}
is the unique element of the minimal Fr\'{e}chet sub-differential of $\mF$, where
\begin{align*}
\partial^0 N\star\rho(x)=\int_{\{x\neq y\}}\sign(x-y)\, \d \rho(y), \ and\ \ \partial^0 N\star\eta(x)=\int_{\{x\neq y\}}\sign(x-y)\, \d \eta(y).
\end{align*}
\end{proposition}
\begin{proof}
  The geodesic convexity of $\mF$ on $\mptra\times \mptra$ is the consequence of two observations. First, the cross-interaction part is geodesically convex as the interspecific interaction potential is given by $N(x)=|x|$, a convex function. Second, the geodesic convexity of the intraspecific self-interaction part can be proven using a nice monotonicity property of the transport map between two measures in $\mptra$ in one dimension. Indeed, in \cite[Lemma 1.4]{CFP12} the authors prove that, given $\mu,\nu\in\mptra$, the transport map $T=T_\mu^\nu$ is essentially non-decreasing, i.e. it is non-decreasing except on a $\mu$-null set. Hence, we can prove $\mathcal{S}[\mu]:=-\frac12 \int_\R N\star\mu\, \d \mu$ is geodesically convex. More precisely, if $T$ is the optimal transport map between $\mu$ and $\nu$, then $g_t=((1-t)\mbox{id}+tT)_\#\mu$ is the geodesic connecting $\mu$ and $\nu$. In particular, by using the mentioned monotonicity property for $T$, we have
\begin{align*}
\mathcal{S}[g_t]&=-\frac{1}{2}\iint_{\mathbb{R}^2}|x-y|\,dg_t(y)\,dg_t(x)\\
&=-\frac{1}{2}\iint_{\mathbb{R}^2}|(1-t)(x-y)+t(T(x)-T(y))|\,d\mu(y)\,d\mu(x)\\
&=-\frac{1}{2}(1-t)\iint_{\mathbb{R}^2}|x-y|\,d\mu(y)\,d\mu(x)-\frac{1}{2}\iint_{\mathbb{R}^2}|T(x)-T(y)|\,d\mu(y)\,d\mu(x)\\
&=(1-t)S[\mu]+tS[\nu],
\end{align*}
whence we get geodesic convexity for $S$, and for $\mathcal{F}$ as well. Obviously, 0-geodesic convexity implies $\lambda$-geodesic convexity for any $\lambda\le0$. In order to prove formula \eqref{eq:subdiff_sign}, let us notice the functional can be written as
$$
\mathcal{F}[\rho,\eta]=\mathcal{S}[\rho]+\mathcal{S}[\eta]+\mathcal{K}[\rho,\eta],
$$
where $\mathcal{S}$ has been introduced above and $\mathcal{K}[\rho,\eta]:=\int_\R N\star\eta\, \d \rho$. Now, thanks to \cite[Theorem 5.1]{BCDFP} and \cite[Proposition 4.3.3]{B} we know $\partial^0S[\rho]=-\partial^0 N\star\rho$, while \cite[Proposition 3.1]{DFF} yields
\begin{equation*}
\partial^0\mathcal{K}[\rho,\eta]=\begin{pmatrix}
\partial^0N\star\eta\\
\partial^0N\star\rho
\end{pmatrix}.
\end{equation*}
In particular, it is easy to check the vector field
\begin{equation*}
\begin{pmatrix}
-\partial^0 N\star\rho+\partial^0N\star\eta\\
-\partial^0 N\star\eta+\partial^0N\star\rho
\end{pmatrix}
\end{equation*}
is an element of the Fr\'{e}chet sub-differential of $\mF$, and it is the unique one of minimal $L^2$-norm by arguing as in \cite[Proposition 2.2]{CDFFLS}.
\end{proof}
\begin{remark}
We highlight that in the presence of atomic parts for $\rho$ or $\eta$ the sub-differential may be empty, as shown in \cite{BCDFP}.
\end{remark}

Recall that, for $\mu,\nu\in\mptra\times\mptra $, the \emph{slope} of a functional $\mF$ on $\mptra\times\mptra$ is defined as
$$
|\partial\mF|[\mu]:=\limsup_{\nu\to\mu}\frac{(\mF(\mu)-\mF(\nu))^+}{\mW_2(\nu,\mu)},
$$
\note{we corrected the formula for the metric slope}and it can be written as
$$
|\partial\mF|[\mu]=\min\{\|\nu\|_{L^2(\mu)}\,|\,\nu\in\partial\mF(\mu)\},
$$
under certain conditions, cf. \cite[Chapter 10]{AGS}.
\begin{definition}\label{def:maximalslope}
An absolutely continuous curve $\gamma(t):[0,T]\to\mptra\times\mptra$ is a \emph{curve of maximal slope} for the functional $\mF$ if the map $t\mapsto\mF(\gamma(t))$ is an absolutely continuous function and the following inequality holds
\begin{equation*}
	\mF(\gamma(s))-\mF(\gamma(t)) \geq \frac{1}{2}\int_s^t \left[|\gamma'|^2(\tau)+|\partial\mF|[\gamma(\tau)]^2\right]\, \d \tau,
\end{equation*}
for all $0\le s\le t\le T$.
\end{definition}

In order to construct a solution to system \eqref{eq:full_system} in the sense of Definition \ref{def:gradflow} we follow the strategy proposed in \cite{AGS} and used in \cite{CDFFLS,DFF}. First, we prove the existence of a curve of maximal slope by means of the so-called ``\textit{Minimizing Movement Scheme}'', cf. \cite{AGS,DeGiorgi}, or \textit{Jordan-Kinderlehrer-Otto scheme}, cf. \cite{JKO}. Then, we prove  the limit curve of the scheme is absolutely continuous w.r.t. the Lebesgue measure provided the initial datum is in $L^m(\R)\times L^m(\R)$ for some $m>1$.

Let $\tau>0$ be a fixed time step and let $\gamma_0=(\rho_0,\eta_0)\in\mptra\times\mptra$ be a fixed initial datum such that $\mF(\gamma_0)<+\infty$. We define a sequence $\{\gamma_\tau^n\}_{n\in\mathbb{N}}=\{(\rhotn,\etatn)\}_{n\in\mathbb{N}}$ recursively. We set $\gamma_\tau^0=\gamma_0$ and, for a given $\gamma_\tau^n\in\mptr\times\mptr$ with $n\geq 0$, we choose $\gamma_\tau^{n+1}$ as
{\color{black}
\begin{equation}
\label{eq:jko}
\gamma_\tau^{n+1}\in\argmin_{\gamma\in\mptr^2}\left\{\frac{1}{2\tau}\mW_2^2(\gamma_\tau^n,\gamma)+\mF(\gamma)\right\}.
\end{equation}
}
Note that \eqref{eq:jko} is well-posed arguing as in \cite[Lemma 2.3 and Proposition 2.5]{CDFFLS} for each component.
Next we define the piecewise constant interpolation of the sequence $\{\gamma_\tau^n\}_{n\in\mathbb{N}}$. Let $T>0$ be fixed and let $N:=\left[\frac{T}{\tau}\right]$. We set
$$
\gamma_\tau(t) = \gamma_\tau^n,
$$
for $t\in((n-1)\tau,n\tau]$. We proceed by showing that the family $\{\gamma_\tau\}_{\tau>0}$ admits a limiting curve and conclude by identifying this limit as a distributional solution to system \eqref{eq:full_system}. The proof in Proposition \ref{prop:narrowcompactness} follows a, by now, classical argument of \cite[Chapter 3]{AGS}, with only minor issues related to some moment estimates in our case. We present it here for the reader's convenience.

\begin{proposition}[Narrow compactness]\label{prop:narrowcompactness}
There exists an absolutely continuous curve \replace{$\gamma: [0,T]\rightarrow\mptra^2$}{$\gamma: [0,T]\rightarrow\mptr^2$} such that the family of piecewise constant interpolations, $\{\gamma_\tau\}_{\tau>0}$ admits a subsequence $\{\gamma_k\}_{k\in \N}:=\{\gamma_{\tau_k}\}_{k\in \N}$ which converges narrowly to $\gamma$ uniformly in $t\in[0,T]$ as $k\rightarrow +\infty$.
\end{proposition}
\begin{proof}
Consider two consecutive iterations, $\gamma_\tau^n$ and $\gamma_\tau^{n+1}$, of the JKO scheme, Eq. \eqref{eq:jko}. By the optimality of $\gamma_\tau^{n+1}$ we obtain
\begin{equation}\label{eq:fromscheme1}
\frac{1}{2\tau}\mW_2^2(\gamma_\tau^n,\gamma_\tau^{n+1})\le\mF(\gamma_\tau^n)-\mF(\gamma_\tau^{n+1}),
\end{equation}
which implies
\begin{equation}\label{eq:estimatefunct1}
\mF(\gamma_\tau^n)\le\mF(\gamma_0),
\end{equation}
for all $n\in\N$. Summing over $k$ from $m$ to $n$ with $m<n$ we obtain the following telescopic sum
\begin{equation}\label{eq:telescopic}
\frac{1}{2\tau}\sum_{k=m}^{n}\mW_2^2(\gamma_\tau^k,\gamma_\tau^{k+1})\le\mF(\gamma_\tau^m)-\mF(\gamma_\tau^{n+1}).
\end{equation}
Now, let us consider $t\in((n-1)\tau,n\tau]$; using the estimate \eqref{eq:telescopic} we obtain
\begin{equation}\label{eq:estimate1formoment}
\mW_2^2(\gamma_\tau(0),\gamma_\tau(t))\le2n\tau\mF(\gamma_0)-2n\tau\mF(\gamma_\tau^n).
\end{equation}
Using the Remark \ref{rem:inequalitymom} after H\"{o}lder and (weighted) Young inequalities, it is possible to obtain a bound from below for $\mF(\gamma_\tau^n)$, which gives in combination with \eqref{eq:estimate1formoment} the following estimate
\begin{equation}\label{eq:estimate2formoment}
\mW_2^2(\gamma_\tau(0),\gamma_\tau(t))\le4T\mF(\gamma_0)+m_2(\gamma_0)+16T^2=:C(\gamma_0,T).
\end{equation}
From  estimate \eqref{eq:estimate2formoment} and the inequality in Remark \ref{rem:inequalitymom} we can deduce  the second moment of $\gamma_\tau(t)$ is uniformly bounded on compact time intervals. Moreover, as a consequence of estimates \eqref{eq:telescopic} and \eqref{eq:estimate2formoment}, using once again the bound from below for $\mF(\gamma_\tau^n)$ we get
\begin{equation}\label{eq:lasttelescopicest}
\sum_{k=m}^{n-1}\mW_2^2(\gamma_\tau^k,\gamma_\tau^{k+1})\le\tau\bar{C}(\gamma_0,T).
\end{equation}
Now, let us consider $0\le s<t$ such that $s\in((m-1)\tau,m\tau]$ and $t\in((n-1)\tau,n\tau]$. It easy to check that $|n-m|<\frac{|t-s|}{\tau}+1$. By means of the Cauchy-Schwartz inequality and \eqref{eq:lasttelescopicest} we obtain $\frac{1}{2}$-H\"{o}lder equi-continuity for $\gamma_\tau$ (up to a negligible error of order $\sqrt{\tau}$) since
\begin{equation}\label{eq:holdercont}
\begin{split}
\mW_2(\gamma_\tau(s),\gamma_\tau(t))&\le\sum_{k=m}^{n-1}\mW_2(\gamma_\tau^k,\gamma_\tau^{k+1})\le\left(\sum_{k=m}^{n-1}\mW_2^2(\gamma_\tau^k,\gamma_\tau^{k+1})\right)^{\frac{1}{2}}|n-m|^{\frac{1}{2}}\\&\le c\left(\sqrt{|t-s|}+\sqrt{\tau}\right),
\end{split}
\end{equation}
where $c$ is a positive constant.
The refined version of the Ascoli-Arzel\`{a} theorem yields the narrow compactness, cf. \cite[Proposition 3.3.1]{AGS}. This completes the proof.
\end{proof}

\note{we changed the spaces from $\mptra$ to $\mptr$ in the JKO scheme above and adjusted the statement of the remark accordingly.}
\begin{remark}[Extension of solutions]
From Proposition \ref{prop:narrowcompactness} we can construct a curve $\gamma:[0,T]\rightarrow \mptr^2$ for any $T>0$. Thus we may extend the solution up to the point where the  second order moments of the solution become unbounded. By construction this is not possible and can only happen at $T=+\infty$. As a consequence Proposition \ref{prop:narrowcompactness} proves the existence of a limiting curve $\gamma:[0,\infty)\rightarrow \mptr^2$.
\end{remark}

As already mentioned, we can prove more refined estimates for the solution $\gamma$ to system \eqref{eq:full_system}. Indeed, we show that, starting with initial data $\rho_0,\eta_0\in\mptr\cap L^m(\R)$, for $m\in (1,+\infty]$, the solution keeps this regularity for every $t\ge0$ and it has second order moments uniformly bounded in time. These properties can be establish by using the ``flow interchange'' technique developed by Otto in \cite{otto}, and Matthes, McCann and Savar\'{e} in \cite{MMCS}. This technique is based on the idea that \textit{the dissipation
of one functional along the gradient flow of another functional equals the dissipation of the second
functional along the gradient flow of the first one}. In this spirit, the ``Evolution Variational Inequality'' (\textbf{E.V.I.}) linked with the auxiliary gradient flow is crucial in order to obtain useful refined estimates (see for instance \cite{DiFranEspFag,DFM}). The connection between gradient flows and evolutionary PDEs of diffusion type shown in \cite{AGS,JKO,otto,S} allows us to consider the (decoupled) system
\begin{subequations}
\begin{equation}\label{eq:heatporsyst}
\begin{cases}
\partial_{t}u_1=\partial_{xx}u_1^m +\varepsilon\partial_{xx}u_1,\\
\partial_{t}u_2=\partial_{xx}u_2^m +\varepsilon\partial_{xx}u_2,
\end{cases}
\end{equation}
as the gradient flow of the functional
\begin{equation}\label{eq:auxiliaryfunctionalE}
\begin{split}
\E(u_1,u_2)=&\frac{1}{m-1}\int_{\R}\left[u_1(x)^m + u_2(x)^m\right]\, \d x\\&+\varepsilon\int_{\R}[u_1(x)\log u_1(x)+u_2(x)\log u_2(x)]\, \d x,
\end{split}
\end{equation}
\end{subequations}
as well as the following system
\begin{subequations}
\begin{equation}\label{eq:linfokkplanksys}
\begin{cases}
\partial_{t}u_1=\partial_{x}(2xu_1) +\varepsilon\partial_{xx}u_1,\\
\partial_{t}u_2=\partial_{x}(2xu_2) +\varepsilon\partial_{xx}u_2,
\end{cases}
\end{equation}
which can be seen as the gradient flow of the functional
\begin{equation}\label{eq:auxiliaryfunctionalG}
\begin{split}
\G(u_1,u_2)=&\int_{\R}|x|^2(u_1(x)+u_2(x))\, \d x\\&+\varepsilon\int_{\R}[u_1(x)\log u_1(x)+u_2(x)\log u_2(x)]\, \d x,
\end{split}
\end{equation}
\end{subequations}
for $\varepsilon>0$ and $m\in(1,\infty)$, with respect to the product 2-Wasserstein distance $\mW_2$. We shall employ the flow interchange strategy twice taking as auxiliary functional
\begin{enumerate}
\item $\E$ to get $L^m$-regularity ($m>1$) for the solution $\gamma$,
\item $\G$ in order to obtain a uniform bound in time for the second order moments of $\gamma$.
\end{enumerate}
For the reader's convenience we shall sometimes use the symbol $\A$ to denote either $\G$ or $\E$.
The functional $\A\in\{\G,\E\}$ possess a $0$-flow given by the semigroup  and $\bm{S}_{\A}=(\sau,\sad)$, see for instance \cite{DS}. In particular, by setting 
{
\color{black}
    $$
    \saut(\nu_1):=u_1(t,\cdot),\quad \text{and}\quad \sadt(\nu_2):=u_2(t,\cdot),
    $$
}
we have $u(t,\cdot)=(u_1(t,\cdot),u_2(t,\cdot))$ is the unique classical solution at time $t$ of system \eqref{eq:heatporsyst} (respectively \eqref{eq:linfokkplanksys}) coupled with an initial value $(\nu_1,\nu_2)$ at $t=0$ in case $\A=\E$ ($\A=\G$, respectively). \note{the notation is slightly abusive, but not worthwhile changing everything.}
\begin{remark}\label{rem:controlbelowentropy}
\replace{From the seminal work of Jordan, Kinderlehrer, and Otto, we know that the entropy}{As in \cite[Proposition 4.1]{JKO}, we know that the log-entropy}
$$
	\mathcal{H}(\rho)=\int_{\Rd}\rho(x)\log\rho(x)\,\d x,
$$
is bounded from below in terms of the second moment $m_2(\rho)$, i.e.,
$$
\mathcal{H}(\rho)\ge -C(m_2(\rho) + 1)^{\beta},
$$
for every $\rho\in\mpta(\Rd)$, $\beta\in(\frac{d}{d+2},1)$, and $C<+\infty$, depending only on the space dimension $d$. We are going to use this inequality in order to have a uniform bound from below for the entropic part in \eqref{eq:auxiliaryfunctionalE} and \eqref{eq:auxiliaryfunctionalG}.
\end{remark}
For every $\gamma=(\rho,\eta)\in\mptra^2$, let us define the dissipation of $\mF$ along $\bm{S}_{\A}$ by
$$
\bm{D}_{\A}\mF(\gamma):=\limsup_{s\downarrow0}\frac{\mF(\gamma)-\mF(\bm{S}_{\A}^s\gamma)}{s},
$$
where $\A\in\{\G,\E\}$. We prove the following proposition.
\begin{proposition}\label{prop:lpestimatest}
Let $m\in (1,+\infty)$ and let $\gamma_0=(\rho_0,\eta_0)\in(\mptra\cap L^m(\R))^2$ be such that $\E(\gamma_0)<+\infty$. The piecewise constant interpolation $\gammat=(\rhot,\etat)$ satisfies
\begin{equation}
\label{eq:linflmbound}
\|\rhot\|_{L^\infty(0,+\infty;L^m(\R))} + \|\etat\|_{L^\infty(0,+\infty;L^m(\R))} \le \|\rho_0\|_{L^m(\R)}+\|\eta_0\|_{L^m(\R)}
\end{equation}
Moreover, the limit curve $\gamma$ belongs to $L^\infty(0,+\infty;L^m(\R))^2$. In fact, this property can be extended to the cases $m=+\infty$.
\end{proposition}
\begin{proof}
As an easy consequence of the definition of the sequence $\{\gammatn\}_{n\in\mathbb{N}}$, for all $s>0$ we have that
$$
\frac{1}{2\tau}\mW_2^2(\gammatnn,\gammatn)+\mF(\gammatnn)\le\frac{1}{2\tau}\mW_2^2(\bm{S}_{\E}^s\gammatnn,\gammatn)+\mF(\bm{S}_{\E}^s\gammatnn).
$$
Dividing by $s>0$ and passing to the $\limsup$ as $s\downarrow0$ we get
\begin{equation}\label{eq:controldiss}
\tau\bm{D}_{\E}\mF(\gammatnn)\le\frac{1}{2}\frac{d^+}{dt}\bigg(\mW_2^2(\bm{S}_{\E}^t\gammatnn,\gammatn)\bigg)\Big|_{t=0}\overset{\bm{(E.V.I.)}}{\le}\E(\gammatn)-\E(\gammatnn),
\end{equation}
where in the last inequality the well-known connection between displacement convexity and the \textbf{E.V.I.} is crucial, see e.g. \cite{DS}. Now, concerning the left-hand side of \eqref{eq:controldiss}, we notice that
\begin{equation}\label{eq:integralformofdis}
\begin{split}
 \bm{D}_{\E}\mF(\gammatnn)&=\limsup_{s\downarrow0}\frac{\mF(\gammatnn)-\mF(\bm{S}_{\E}^s\gammatnn)}{s}\\&=\limsup_{s\downarrow0}\int_0^1\left(-\frac{d}{dz}\Big|_{z=st}\mF(\bm{S}_{\E}^{z}\gammatnn)\right)\,dt.
\end{split}
\end{equation}
Hence, let us focus on the time derivative inside the above integral. Keep in mind that $\bm{S}_{\E}^t\gammatnn$ is the solution to the decoupled system of nonlinear parabolic equations with strictly positive coefficients, system \eqref{eq:heatporsyst}. Then, using the $C^\infty$-regularity of $\bm{S}_{\E}^t\gammatnn$ we may infer
\begin{equation}\label{eq:derivfuncE}
	\begin{split}
		\ddt \mF(\bm{S}_{\E}^t\gammatnn) =& -\int_\R \left( \seut \rhotnn - \sedt\etatnn \right)
        \left( [\seut\rhotnn]^m - [\sedt\etatnn]^m
        \right) \, \d x\\
&-\varepsilon\int_\R
 	\left( \seut \rhotnn-\sedt\etatnn \right)^2\, \d x,\\
	\end{split}
\end{equation}
where the terms at infinity in the integration by parts vanish due to the rapid decay of the solution to a nondegenerate diffusion equation \cite{LSU}. \eqref{eq:derivfuncE} yields
\begin{equation}
	\label{eq:dissipation}
	\ddt \mF(\bm{S}_{\E}^t\gammatnn) \leq 0.
\end{equation}
Combining \eqref{eq:dissipation} with \eqref{eq:integralformofdis} and \eqref{eq:controldiss} we obtain
\begin{equation*}
0\le\tau\bm{D}_{\E}\mF(\gammatnn)\le\E(\gammatn)-\E(\gammatnn),
\end{equation*}
whence
\begin{equation}
\E(\gammatn)\le\E(\gamma_0),
\end{equation}
for all $n\in\mathbb{N}$. By Remark \ref{rem:controlbelowentropy} we  control the log-entropic part of $\E$ and we deduce that, as $\varepsilon\downarrow0$,
\begin{equation}
	\int_\R \left[\rhotn(x)\right]^m + \left[ \etatn(x) \right]^m \, \d x \leq \int_\R \left[ \rho_0(x) \right]^m + \left[ \eta_0(x) \right]^m \, \d x,
\end{equation}
whence
\begin{equation}
	\int_\R \left[ \rhot(t,x) \right]^m + \left[ \etat(t,x) \right]^m\,  \d x \leq \int_\R \left[ \rho_0(x) \right]^m + \left[ \eta_0(x)\right]^m \, \d x,
\end{equation}
for every $t \geq 0$. This proves estimate \eqref{eq:linflmbound} for $m\in (1, \infty)$. In order to extend the estimate to the case $m=+\infty$, we observe that
\begin{align*}
  \|\rhot(t,\cdot)\|_{L^\infty(\R)}+\|\etat(t,\cdot)\|_{L^\infty(\R)}&\leq \limsup_{m\rightarrow +\infty}\left[\|\rhot(t,\cdot)\|_{L^m(\R)}+\|\etat(t,\cdot)\|_{L^m(\R)}\right]\\
  & \leq \limsup_{m\rightarrow +\infty}\left[\|\rho_0\|_{L^m(\R)}+\|\eta_0\|_{L^m(\R)}\right]\\
  &\leq \limsup_{m\rightarrow +\infty}\left[\|\rho_0\|_{L^\infty(\R)}^{\frac{m-1}{m}}\|\rho_0\|_{L^1(\R)}^{\frac{1}{m}}+\|\eta_0\|_{L^\infty(\R)}^{\frac{m-1}{m}}
  \|\eta_0\|_{L^1(\R)}^{\frac{1}{m}}\right]\\
  & = \|\rho_0\|_{L^\infty(\R)} + \|\eta_0\|_{L^\infty(\R)}.
\end{align*}
We conclude that, for all $T>0$, the subsequence $\{\gamma_{\tau_k}\}_{k\in\mathbb{N}}$ obtained from Proposition \ref{prop:narrowcompactness} is uniformly bounded in $L^\infty([0,T];L^m(\R))^2$. By Banach-Alaoglu's Theorem, in case $m$ is finite, there exists a subsequence $(\tau_k')\subset(\tau_k)$ such that $\{\gamma_{\tau_k'}\}_{k\in\mathbb{N}}$ converges in the weak $L^m_{x,t}$ topology to some limit $\gamma'\in L^m([0,T]\times \R)^2$. In the case of $m=+\infty$ the above subsequence exists in the weak-$\star$ topology of $L^\infty([0,T]\times \R)$. Due to Proposition \ref{prop:narrowcompactness} the limit $\gamma'$ coincides with $\gamma$ on $[0,T]$. By a simple weak lower semi-continuity argument we deduce that $\gamma$ inherits the same estimates as the approximating sequence $\gamma_{\tau_k}$. Since \replace{The}{} $T$ was arbitrary we conclude the proof.
\end{proof}

\begin{lemma}\label{lem:momunifdounded}
Let $\gamma_0=(\rho_0,\eta_0)\in \mptra\times\mptra$ be such that $\G(\gamma_0)<+\infty$. The piecewise constant interpolation $\gammat=(\rhot,\etat)$ satisfies
\begin{equation}\label{eq:linfmom}
\int_\R|x|^2[\rhot(t,x)+\etat(t,x)]\, \d x \leq \int_\R|x|^2[\rho_0(x)+\eta_0(x)]\, \d x,
\end{equation}
for every $t\ge0$.In addition, the limiting curve $\gamma$ has uniformly bounded second order moments in time.
\end{lemma}
\begin{proof}
Arguing as in the proof of Proposition \ref{prop:lpestimatest}, from the scheme \eqref{eq:jko} we easily get
\begin{equation}\label{eq:controldissmom}
\tau\bm{D}_{\G}\mF(\gammatnn) \leq \frac12 \frac{\d^+}{\d t}\bigg(\mW_2^2(\bm{S}_{\G}^t\gammatnn,\gammatn)\bigg)\Big|_{t=0}\overset{\bm{(E.V.I.)}}{\le}\G(\gammatn)-\G(\gammatnn),
\end{equation}
where, again, the left-hand side of \eqref{eq:controldissmom} can be rewritten as
\begin{equation}\label{eq:integralformofdismom}
 \bm{D}_{\G}\mF(\gammatnn)=\limsup_{s\downarrow0}\int_0^1\left(-\frac{\d}{\d z}\Big|_{z=st}\mF(\bm{S}_{\G}^{z}\gammatnn)\right)\, \d t.
\end{equation}
Since $\bm{S}_{\G}^t\gammatnn$ is the solution to system \eqref{eq:linfokkplanksys}, which is a decoupled system of linear Fokker-Planck equations, we use its $C^\infty$-regularity to obtain
\begin{equation}\label{eq:derivfuncG}
\begin{split}
	\ddt\mF(\bm{S}_{\G}^t\gammatnn) =
    &-\int_\R|N'\star(\sgut\rhotnn-\sgdt\etatnn)|^2\, \d x\\
    &-\varepsilon\int_\R(\sgut\rhotnn-\sgdt\etatnn)^2\, \d x\\
    &+\frac{1}{2}\bigg[ N\star \left(\sgut\rhotnn-\sgdt\etatnn\right) \, N'\star \left( \sgut\rhotnn-\sgdt\etatnn\right) \bigg]_{x=-\infty}^{x=+\infty}\\
    &+\varepsilon \bigg[N'\star \left(\sgut\rhotnn-\sgdt\etatnn\right) \,  \left(\sgut\rhotnn-\sgdt\etatnn\right)\bigg]_{x=-\infty}^{x=+\infty}\\
    &-\bigg[2xN\star \left(\sgut\rhotnn-\sgdt\etatnn\right) \,\left(\sgut\rhotnn-\sgdt\etatnn \right) \bigg]_{x=-\infty}^{x=+\infty}\\
    &-\varepsilon\bigg[N\star \left(\sgut\rhotnn-\sgdt\etatnn\right) \, \partial_x \left(\sgut\rhotnn-\sgdt\etatnn\right)\bigg]_{x=-\infty}^{x=+\infty}.
\end{split}
\end{equation}
Let us consider the boundary terms individually. For convenience we set $\rho^t=\sgut\rhotnn$, $\eta^t=\sgdt\etatnn$, and $\kappa:=\rho^t-\eta^t$ in order to simplify the notation. Concerning the first term, we have
\begin{equation}\label{eq:boundterm1}
\begin{split}
	\left| N\star\kappa\, N'\star\kappa\right|
	&=\left|\,\int_\R|x-y|\kappa(y)\, \d y\,  \int_\R\sign(x-y)\kappa(y)\, \d y\right|\\
    &=\left|\int_\R(x-y)\kappa(y)\, \d y + 2\int_x^{+\infty}(y-x)\kappa(y)\, \d y\right| \times \left|\int_\R \kappa(y)\, \d y-2\int_x^{+\infty}\kappa(y)\, \d y\right|\\
    &=\left|-\int_\R y\kappa(y)\, \d y-2x\int_x^{+\infty}\kappa(y)\, \d y+2\int_x^{+\infty}y\kappa(y)\, \d y\right|\times\left|-2\int_x^{+\infty}\kappa(y)\, \d y\right|\\&\le2\left( 2+2m_2(\rho^t)+2m_2(\eta^t) + \frac{2}{|x|} (m_2(\rho^t) + m_2(\eta^t))\right)\frac{2\left[m_2(\rho^t)+m_2(\eta^t)\right]}{|x|^2},
\end{split}
\end{equation}
which vanishes as $|x|\to+\infty$. Regarding the second term we note that
\begin{equation}\label{eq:boundterm2}
\begin{split}
	\left|N'\star\kappa\, \kappa\right|&=\left|\int_\R\sign(x-y)\kappa(y)\kappa(x)\, \d y\right|\\
	&\le2(\rho^t(x)+\eta^t(x))
\end{split}
\end{equation}
which vanishes as $|x|\to+\infty$ because $\rho^t$ and $\eta^t$ are solutions of linear Fokker-Planck equations decaying exponentially at infinity. Now, for the same reason, there holds
\begin{align}
	\label{eq:boundterm3}
	\begin{split}
		\left|2xN\star\kappa\cdot \kappa\right|
    	&= \left|\int_\R2x|x-y|\kappa(y)\kappa(x)\, \d y\right| \\
    	&\le 2\int_\R|x|^2(\rho^t(y)+\eta^t(y))(\rho^t(x)+\eta^t(x))\, \d y\\
    	&\quad +2\int_\R|x| |y|(\rho^t(y)+\eta^t(y))(\rho^t(x)+\eta^t(x))\, \d y\\
    	&\le 4|x|^2(\rho^t(x)+\eta^t(x))+(4+m_2(\rho^t)+m_2(\eta^t))|x|(\rho^t(x)+\eta^t(x)),
	\end{split}
\end{align}
vanishes $|x|\to+\infty$. As for the last boundary term we get
\begin{align}\label{eq:boundterm4}
	\begin{split}
		|N\star\kappa\,\partial_x\kappa|
    	&=\left|\int_\R|x-y|\kappa(y)\partial_x\kappa(x) \d y\right|\\\
    	&\le\int_\R|x|(\rho^t(y)+\eta^t(y))(|\partial_x\rho^t(x)|+|\partial_x\eta^t(x)|)\, \d y\\
		&\quad +\int_\R|y|(\rho^t(y)+\eta^t(y))(|\partial_x\rho^t(x)|+|\partial_x\eta^t(x)|)\, \d y\\
		&\le2|x|(|\partial_x\rho^t(x)|+|\partial_x\eta^t(x)|)+(2+m_2(\rho^t)+m_2(\eta^t))(|\partial_x\rho^t(x)|+|\partial_x\eta^t(x)|),
	\end{split}
\end{align}
which, again, goes to 0 as $|x|\to+\infty$. Using \eqref{eq:boundterm1}, \eqref{eq:boundterm2}, \eqref{eq:boundterm3}, and \eqref{eq:boundterm4} in \eqref{eq:derivfuncG} we deduce that
\begin{equation}
	\ddt \mF(\bm{S}_{\G}^t\gammatnn)\le0,
\end{equation}
which, in combination with \eqref{eq:integralformofdismom} and \eqref{eq:controldissmom}, gives
$$
	\G(\gammatn)\le\G(\gamma_0).
$$
Hence, taking into account Remark \ref{rem:controlbelowentropy}, letting $\varepsilon\to0^+$ we get
\begin{equation}
	\int_\R |x|^2 [\rhot(t,x)+\etat(t,x)] \, \d x \le \int_\R|x|^2[\rho_0(x)+\eta_0(x)]\, \d x,
\end{equation}
for every $t\ge0$, i.e. estimate \eqref{eq:linfmom}. By similar considerations to the ones at the end of the proof of Proposition \ref{prop:lpestimatest} and from Proposition \ref{prop:narrowcompactness} we conclude that $\gamma$ has second order moments uniformly bounded in time.
\end{proof}
{\color{black}
\begin{remark}[Preservation of absolute continuity]
    A natural question is to ask as to whether absolute continuity of solutions is kept provided that the initial data satify $\rho_0, \eta_0 \in L^1(\R)\cap\mptr$. The answer to this question is positive. In fact, $\rho_0, \eta_0\in L^1(\R)\cap\mptr$ implies the existence of two nonnegative, superlinear, and convex functions $\Phi_1, \Phi_2$ with $\Phi_i(0)=0$, for $i=1,2$, satisfying $\Phi_1(\rho_0), \Phi_2(\eta_0)\in L^1(\R)$. It is easy to check that, individually, the two $\Phi_i$ are geodesically convex as they satisfy the McCann condition trivially in one dimension. Setting $\Phi:=\sup_{i=1,2}(\Phi_i)$, we readily verify that this function satisfies the McCann condition and is therefore geodesically convex. Moreover, by this choice, $\Phi(\rho_0), \Phi(\eta_0)\in L^1(\R)$. Applying the flow interchange argument as above for the extended functional
    $$
        \curlyE(\nu_1, \nu_2) = \int_\R \Phi(\nu_1) + \Phi(\nu_2) \d x +\epsilon \mathcal{H}(\nu_1) + \epsilon\mathcal{H}(\nu_2),
    $$
    yields uniform bounds of the form
    \begin{align*}
        \int_\R \Phi(\rho_\tau^{n+1}) + \Phi(\eta_\tau^{n+1}) \d x \leq \int_\R  \Phi(\rho_0) + \Phi(\eta_0) \d x.
    \end{align*}
    The uniform control of superlinear function $\Phi$ composed with the two species then yields uniform bound on $(\rho_\tau^n)_{\tau>0, n \in \N}, (\eta_\tau^n)_{\tau>0, n \in \N}$ in $L^1(\R)$. Together with the uniform control of the second order moments we may invoke the Dunford-Pettis theorem to obtain weak compactness in $L^1(\R)$.
\end{remark}
}

The final step in this procedure is to prove the curve $\gamma$ obtained in Proposition \ref{prop:narrowcompactness} is a curve of maximal slope for $\mF$, arguing as in \cite{CDFFLS,DFF}. Since curves of maximal slope coincide with gradient flows, see \cite[Theorem 11.1.3]{AGS}, we actually have that $\gamma$ is a gradient flow solution to \eqref{eq:full_system} in the sense of Definition \ref{def:gradflow}. Let us summarise the procedure for the sake of completeness. We denote by $\tilde{\gamma}_\tau$ the \textit{De Giorgi variational interpolation} (cf. \cite[Definition 3.2.1]{AGS}), i.e. any interpolation $\tilde{\gamma}_\tau:[0,+\infty)\to\mptr^2$ of the discrete values $\{\gamma_\tau^n\}_{n\in\N}$ defined through scheme \eqref{eq:jko} such that
\begin{equation}\label{eq:degiorgiinterpolation}
\tilde{\gamma}_\tau(t)=\tilde{\gamma}_\tau((n-1)\tau+\delta)\in\argmin_{\gamma\in\mptr^2}\left\{\frac{1}{2\delta}\mW_2^2(\gamma_\tau^{n-1},\gamma)+\mF(\gamma)\right\},
\end{equation}
if $t=(n-1)\tau+\delta\in((n-1)\tau,n\tau]$. Now, from \cite[Theorem 3.14, Lemma 3.2.2]{AGS} it is possible to get the following energy inequality
\begin{equation}\label{eq:energyinequality}
	\mF(\gamma_0)\ge\frac{1}{2}\int_0^T\|v_k(t)\|_{L^2(\gamma_k(t))}^2\, \d t+\frac{1}{2}\int_0^T|\partial\mF|^2[\tilde{\gamma}_k(t)]\, \d t+\mF(\gamma_k(T)),
\end{equation}
where the pair $(\gamma_k,v_k)$ is the solution of the continuity equation $\partial_t\gamma_k(t)+\text{div}(v_k(t)\gamma_k(t))=0$ in the sense of distributions. Here $\gamma_k$ is the subsequence from Proposition \ref{prop:narrowcompactness} and $v_k(t)$ is the unique velocity field with minimal $L^2(\gamma_k(t))$-norm (see Remark below), and $\tilde{\gamma}_k:=\tilde{\gamma}_{\tau_k}$ is defined by \eqref{eq:degiorgiinterpolation}.
\begin{remark}[Absolutely continuous curve and the continuity equation]
Thanks to Proposition \ref{prop:narrowcompactness} we know $\gamma_k$ is an absolutely continuous curve, therefore we can identify its tangent vectors with the velocity fields $v_k(t)$ such that the continuity equation $\partial_t\gamma_k(t)+\text{div}(v_k(t)\gamma_k(t))=0$ is satisfied in a distributional sense, according to \cite[Theorem 8.3.1]{AGS}. Furthermore, \cite[Proposition 8.4.5]{AGS} asserts there is only one $v_k(t)$ with minimal $L^2(\gamma_k(t))$-norm, equal to the metric derivative of $\gamma_k(t)$ for a.e. t.
\end{remark}

Up to a subsequence both the interpolations $\gamma_\tau$ and $\tilde{\gamma}_\tau$ narrowly converge to $\gamma$ in view of Proposition \ref{prop:narrowcompactness}. Proving that $\gamma$ is a curve of maximal slope in the sense of Definition \ref{def:maximalslope} is then a consequence of the lower semi-continuity of the slope and the energy inequality \eqref{eq:energyinequality} retracing \cite[Lemma 2.7 and Theorem 2.8]{CDFFLS}. Thanks to \cite[Theorem 11.1.3]{AGS} we actually have that $\gamma$ is a gradient flow solution to \eqref{eq:full_system} in the sense of Definition \ref{def:gradflow}.
At this stage, the uniqueness of the gradient flow solutions in the sense of Definition \ref{def:gradflow} follows from the geodesic convexity of $\mF$ proven in Proposition \ref{prop:geodconvandsubdiff}, relying on \cite[Theorem 11.1.4]{AGS}. More precisely,
given two gradient flow solutions $\gamma_1(t)$ and $\gamma_2(t)$ in the sense of Definition \ref{def:gradflow}, we obtain the stability property
\begin{equation}\label{eq:contraction}
  \mW_2(\gamma_1(t),\gamma_2(t))\le \mW_2(\gamma_{1}(0),\gamma_{2}(0)),
\end{equation}
for all $t\ge0$. In addition, the unique gradient flow solution satisfies the Evolution Variational Inequality (\textbf{E.V.I.}):
\begin{equation}\label{eq:EVI1}
\frac{1}{2}\ddt \mW_2^2(\gamma(t),\bar{\gamma})
\le\mF(\bar{\gamma})-\mF(\gamma(t))
\end{equation}
for almost all $t>0$ and  all $\bar{\gamma}\in\mP(\R)^2$.

The property \eqref{eq:EVI1} can actually be used to show a stronger property, namely that $\gamma$ is a gradient flow in the sense of Definition \ref{def:solutionmeasures}, which implies uniqueness in the weaker notion of solution defined in Definition \ref{def:solutionmeasures} in view of Theorem \ref{thm:solutionmeasures}. We prove this statement in the following Theorem, which also collects all the estimates proven in this subsection.

\begin{theorem}\label{thm:regularity}
  Let $m\in (1,+\infty]$. Let $\rho_0, \eta_0\in \mP_2(\R)\cap L^m(\R)$. Then, there exists a unique $\gamma=(\rho,\eta)\in L^\infty([0,+\infty);\, \mP_2(\R)^2\cap L^m(\R)^2)$ solving \eqref{eq:full_system} in the sense of Definition \ref{def:gradflow}. Moreover, $\gamma$ is the unique solution to \eqref{eq:full_system} in the sense of Definition \ref{def:solutionmeasures} as well. Finally, we have the properties
  \begin{align*}
     \|\rho(t,\cdot)\|_{L^m(\R)} + \|\eta(t,\cdot)\|_{L^m(\R)} &\leq \|\rho_0\|_{L^m(\R)} + \|\eta_0\|_{L^m(\R)}, \\
     \int_\R |x|^2 \rho(t,x) \d x +\int_\R |x|^2 \eta(t,x) \d x &\leq \int_\R|x|^2 \rho_0(x) \d x + \int_\R|x|^2 \eta_0(x) \d x .
  \end{align*}
\end{theorem}

\begin{proof}
All the statements have been proven earlier in this subsection, in particular in Proposition \ref{prop:lpestimatest} and Lemma \ref{lem:momunifdounded}. We only need to prove that $\gamma$ is a solution to \eqref{eq:full_system} in the sense of Definition \ref{def:solutionmeasures}. Recall the \textbf{E.V.I.} \eqref{eq:EVI1} for a general $\overline{\gamma}\in \mP_2(\R)^2$. Integrating the inequality \eqref{eq:EVI1} on the time interval $[s,t]$ and dividing by $t-s$ we obtain
\begin{equation*}
  \frac{1}{2(t-s)}\left[\mW_2^2(\gamma(t),\bar{\gamma})-\mW_2^2(\gamma(s),\bar{\gamma})\right]
\le\frac{1}{t-s}\int_s^t\left[\mF(\bar{\gamma})-\mF(\gamma(t'))\right] \d t'.
\end{equation*}
Consider now the pseudo-inverse variables $X_\rho$, $X_\eta$, $\overline{X}_\rho$, and $\overline{X}_\eta$ of $\rho$, $\eta$, $\overline{\rho}$, and $\overline{\eta}$ respectively, where $\gamma=(\rho,\eta)$ and $\overline{\gamma}=(\overline{\rho},\overline{\eta})$. The formulas \eqref{eq:CoV_PseudoInverse} and \eqref{eq:W2PseudoInverses} applied to our case imply
\begin{align*}
    &\frac{1}{2(t-s)}\left[\|X_\rho(t)-\overline{X}_\rho\|_{L^2}^2+\|X_\eta(t)-\overline{X}_\eta\|_{L^2}^2 -\|X_\rho(s)-\overline{X}_\rho\|_{L^2}^2+\|X_\eta(s)-\overline{X}_\eta\|_{L^2}^2 \right]\\
& \ \le\frac{1}{t-s}\int_s^t\left[\mathfrak{F}(\overline{X}_\rho,\overline{X}_\eta)-\mathfrak{F}(X_\rho(t'),X_\eta(t'))\right] \d t'.
\end{align*}
We notice that the indicator function has been considered as zero in $\mathfrak{F}$ as all the involved variables are in the cone $\mathcal{C}$. By absolute continuity in time of the curve $t\mapsto (X_\rho(t),X_\eta(t))$, recalling the expression of $\mathfrak{F}$, we can let $s\uparrow t$ and obtain
\begin{align*}
  & (\dot{X}_\rho(t),X_\rho(t)-\overline{X}_\rho)_{L^2(0,1)}+(\dot{X}_\eta(t),X_\eta(t)-\overline{X}_\eta)_{L^2(0,1)}\leq \mathfrak{F}(\overline{X}_\rho,\overline{X}_\eta)-\mathfrak{F}(X_\rho(t),X_\eta(t)),
\end{align*}
which, since $(\overline{X}_\rho,\overline{X}_\eta)$ was arbitrary, is equivalent to state that $Z(t)=(X_\rho(t),X_\eta(t))$ satisfies
\[
	-\dot{Z}(t)\in \partial\mathfrak{F}[Z(t)].
\]
The proof will be completed once we show that $Z(t)$ is a Lipschitz curve. Recall the estimate \eqref{eq:estimate1formoment} at the level of  the JKO scheme, which can be rewritten as
\[\mW_2^2(\gamma_\tau(0),\gamma_\tau(h))\leq 2(h+\tau)[\mF(\gamma_0)-\mF(\gamma_\tau(h))],\]
for all $h>0$. Sending $\tau\downarrow 0$ and using the fact that the functional $\mF$ is continuous w.r.t. $\mW_2$, we get
\[\frac{1}{h^2}\mW_2^2(\gamma_0,\gamma(h))\leq \frac{2}{h}[\mF(\gamma_0)-\mF(\gamma(h))].\]
In the pseudo-inverse formalism the above estimate reads
\[\frac{1}{h^2}\left[\|X_\rho(h)-X_0\|_{L^2}^2+\|X_\eta(h)-Y_0\|_{L^2}^2\right]\leq \frac{2}{h}[\mathfrak{F}(X_0,Y_0)-\mathfrak{F}(X_\rho(h),X_\eta(h))],\]
where $X_0$ and $Y_0$ are the pseudo-inverses corresponding to $\rho_0$ and $\eta_0$, respectively. The definition of sub-differential, $\partial \mathfrak{F}$, then implies
\[\lim_{h\downarrow 0}\frac{1}{h^2}\left[\|X_\rho(h)-X_0\|_{L^2}^2+\|X_\eta(h)-Y_0\|_{L^2}^2\right]\leq \lim_{h\downarrow 0}\frac{2}{h}\left[(\tilde{X}_\rho,X_0-X_\rho(h))_{L^2}+(\tilde{X}_\eta,Y_0-X_\eta(h))_{L^2}\right],
\]
for all $(\tilde{X}_\rho,\tilde{X}_\eta)\in \partial \mathfrak{F}[(X_0,Y_0)]$. Hence, we easily get that $\dot{Z}(0)$ is bounded in $L_2(0,1)^2$ by the estimate
\[\|\dot{Z}(0)\|_{L_2^2}\leq 2\|\partial^0 \mathfrak{F}[(X_0,Y_0)]\|_{L_2^2}.\]
Finally, the stability property \eqref{eq:contraction} implies for all $t,h>0$
\[
	\|Z(t+h)-Z(t)\|_{L_2^2}\leq \|Z(h)-Z(0)\|_{L_2^2}.
\]
Upon dividing  by $h$ and letting $h\downarrow 0$ we get
\[
	\|\dot{Z}(t)\|_{L_2^2}\leq\|\dot{Z}(0)\|_{L_2^2},
\]
which gives the desired regularity for $Z(t)$.
\end{proof}

\section{Steady states and minimisers of the energy}\label{sec:steady_states}
In what follows we shall study the energy \eqref{eq:enfunctional} associated to system \eqref{eq:first}. First we shall see that the energy can be written in a completely symmetric way which then allows us to show its boundedness from below by zero.
\begin{align*}
	\mathcal{F}(\rho, \eta) &= -\frac12 \int_\R  \rho N\star \rho \d x - \frac12 \int_\R \eta N\star \eta\d x + \int_\R \rho N\star\eta\d x\\
    &=-\frac12 \iint_{\R^2}N(x-y)\big[\rho(x)\rho(y) -2 \rho(x)\eta(y) + \eta(x)\eta(y)\big] \d y\d x\\
    &=-\frac12 \iint_{\R^2}N(x-y)\big[\rho(x)[\rho(y) - \eta(y)] - [\rho(x) + \eta(x)]\eta(y)\big] \d y\d x.
\end{align*}
As the kernel, $N(x)=|x|$, is symmetric we may swap the roles of $x$ and $y$ in the second term in the integral to obtain
\begin{align*}
\mathcal{F}(\rho, \eta) &=-\frac12 \iint_{\R^2}N(x-y)\big[\rho(x)[\rho(y) - \eta(y)] - [\rho(y) + \eta(y)]\eta(x)\big] \d y\d x\\
    &= -\frac12\int_\R (\rho-\eta) N\star (\rho - \eta)\d x,
\end{align*}
hence the energy does not depend on the individual densities but merely on their difference. By abuse of notation we shall write
\begin{align*}
	\mathcal{F}(\kappa) := -\frac12 \int_\R \kappa N\star\kappa \d x,
\end{align*}
for $\kappa\in L^1((1+|x|^2)\d x)$ with zero mean. We  introduce the set  of $L^1$-functions with finite second order moments with zero mean
\begin{align*}
	L_0^1:= \left\{f\in L^1\big((1+|x|^2)\d x\big)\,\bigg|\, \int_\R f\d x = 0\right\},
\end{align*}
in order to formulate the following Proposition establishing the boundedness of the energy functional.
\begin{proposition}[Characterisation of energy minimisers -- 1]\label{prop:funcpositive}There holds
\begin{align*}
	\mathcal{F}(\kappa) \geq 0,
\end{align*}
for any $\kappa \in L_0^1\big((1+|x|^2)\d x\big)$. Moreover $\mathcal{F} = 0$ if and only if $\kappa = 0$ almost everywhere.
\end{proposition}
\begin{proof}
Let $\kappa \in L^1\big((1+|x|^2)\d x\big)$ be arbitrary. It is well-known that
\begin{align*}
	\kappa = \delta\star \kappa = \frac12 N''\star \kappa,
\end{align*}
where the last equality holds due to the fact that $N/2$ is the fundamental solution of the Laplace equation in one dimension. Thus we may write
\begin{align*}
	\mathcal{F}(\kappa) &= -\frac12 \int_\R \kappa N\star \kappa \d x\\
    &= -\frac12 \int_\R \delta \star \kappa N \star \kappa \d x\\
    &=-\frac14 \int_\R  N''\star \kappa N\star\kappa \d x\\
&= - \bigg[\frac14N'\star \kappa\, N\star \kappa\bigg]_{x=-\infty}^{x=+\infty} + \frac14 \int_\R |N'\star \kappa|^2\d x,
\end{align*}
by an integration by parts. Arguing as in the proof of Proposition \ref{lem:momunifdounded}, the boundary term vanishes. Hence we conclude
\begin{align*}
	\mathcal{F}(\kappa) = \frac14 \int_{\R} \big|N'\star\kappa\big|^2\d x \geq 0.
\end{align*}
Clearly, equality holds if and only if $N'\star \kappa = 0$. Differentiating this expression once yields the second assertion and concludes the proof.
\end{proof}
Notice that the energy is not bounded from below in case of different masses for $\rho$ and $\eta$ as specified in Remark \ref{rem:different-masses-steady-states}.
\begin{proposition}[Energy minimisers are steady states]\label{prop:char_steady}
A pair $(\rho,\eta)\in\mptra\times\mptra$ is an energy minimiser if and only if it is a steady states of system \eqref{eq:full_system}.
\end{proposition}
\begin{proof}
    In view of Proposition \ref{prop:funcpositive} we know that $(\rho,\eta)$ is an energy minimiser if and only if $\rho=\eta$ almost everywhere. Hence, Proposition \ref{prop:geodconvandsubdiff} gives $v_{i}=- \,(\partial^0\mF[\rho,\eta])_i=0$ for $i=1,2$ on  supp($\rho$), whence we conclude that $(\rho,\eta)$ is a stationary state. Now, assume that  $(\rho,\eta)$ is a stationary state of system \eqref{eq:full_system}. From the energy inequality \eqref{eq:energyinequality} we get
    $$
    	\frac{1}{2}\int_0^T\|v(t)\|_{L^2(\gamma)}^2\, \d t+\frac{1}{2}\int_0^T|\partial\mF|^2[\gamma(t)]\, \d t\le0,
    $$
    which implies $|\partial\mF|^2[\gamma]=0$. Thus, in view of Proposition \ref{prop:geodconvandsubdiff} we obtain $\rho=\eta$ almost everywhere. As a consequence of Proposition \ref{prop:funcpositive} we conclude $(\rho,\eta)$ is a minimiser of the energy $\mF$.
\end{proof}
We conclude this section by providing a characterisation for the $\omega$-limit set of a solution to system \eqref{eq:full_system}. For the sake of completeness, let us recall the definition of $\omega$-limit set according to \cite[Definition 9.1.5]{HarCaz}.
\begin{definition}\label{om_lim}
Let $(X,d)$ be a complete metric space and consider a dynamical system $\left\{S_t\right\}_{t\geq 0}$. For $x\in X$ the set
\[
 \omega(x)=\left\{y\in X \,| \,\exists \, t_n\to \infty \mbox{ s.t. } S_{t_n}\left[x\right]\to y,\,\mbox{ as }n\to\infty\right\}
\]
is called $\omega-$limit set of $x$.
\end{definition}

Now, let us state the following Theorem.

\begin{theorem}\label{thm:omega}
Let $\gamma=(\rho,\eta)$ be the solution to system \eqref{eq:full_system} with initial datum $\gamma_0=(\rho_0,\eta_0)\in(\mptra\cap L^m(\R))^2$. Then
$$\omega(\gamma)\subseteq\{(\rho,\eta)\in(\mptra\cap L^m(\R))^2|\rho=\eta\ \text{a.e.}\}.
$$
\end{theorem}
\begin{proof}
Since $\gamma_0=(\rho_0,\eta_0)\in(\mptra\cap L^m(\R))^2$ from Proposition \ref{prop:lpestimatest} we know
$$
\|\gamma\|_{L^\infty(0,+\infty;L^m(\R))^2}\le \|\rho_0\|_{L^m(\R)}+\|\eta_0\|_{L^m(\R)},
$$
whence
$$
\|\gamma(t)\|_{L^m(\R)^2}\le\|\rho_0\|_{L^m(\R)}+\|\eta_0\|_{L^m(\R)},
$$
for a.e. $t>0$. Then we can consider an unbounded, increasing sequence $\{t_n\}_{n\in\mathbb{N}}$ such that $t_n\rightarrow+\infty$ and $\gamma(t_n)\rightharpoonup\tilde \gamma$ weakly in $L^m$, as $n\to+\infty$, where $\tilde\gamma=(\tilde\rho,\tilde\eta)$. According to \cite[Theorem 5.3]{AS07}, since $\gamma$ is a gradient flow solution to system \eqref{eq:full_system} we have
$$
	-\ddt \mF(\gamma(t)) = \|v_t(t)\|_{L^2}^2 = |\partial\mF|^2 [\gamma(t)],
$$
for a. e. $t>0$, whence
$$
	\ddt \mF(\gamma(t))=-|\partial\mF|^2[\gamma(t)]\le0.
$$
Now, if we integrate in a general time interval $(t_n,t_n+1)$, we have
\begin{equation*}
\begin{split}
\int_{t_n}^{t_n+1}\left(-\frac{d}{ds}\mF(\gamma(s))\right)\,\d s=\int_0^1\left(-\frac{d}{ds}\mF(\gamma(s+t_n))\right)\,\d s=\int_0^1|\partial\mF|^2[\gamma(s+t_n)]\,\d s,
\end{split}
\end{equation*}
which gives, passing to the $\liminf$ as $n\to+\infty$,
\begin{equation}\label{eq:slopelimit}
0=\liminf_{n\to+\infty}\int_0^1|\partial\mF|^2[\gamma(s+t_n)]\,\d s\ge\int_0^1|\partial\mF|^2[\tilde \gamma]\,\d s=|\partial\mF|^2[\tilde \gamma],
\end{equation}
by means of the lower semi-continuity of the slope already used in Subsection \ref{subsec:densities_case} (cf. \cite[Lemma 2.7]{CDFFLS}).
Hence, as a trivial consequence of Eq. \eqref{eq:slopelimit} we get $|\partial\mF|^2[\tilde \gamma]=0$, which, according to Proposition \ref{prop:geodconvandsubdiff}, implies that $\partial^0N\star(\tilde\rho-\tilde\eta)=0$ almost everywhere. Thus by differentiating we obtain the result.
\end{proof}

\begin{figure}[ht!]
	\centering
    \subfigure[Evolution.]{
    	\includegraphics[width=0.30\textwidth]{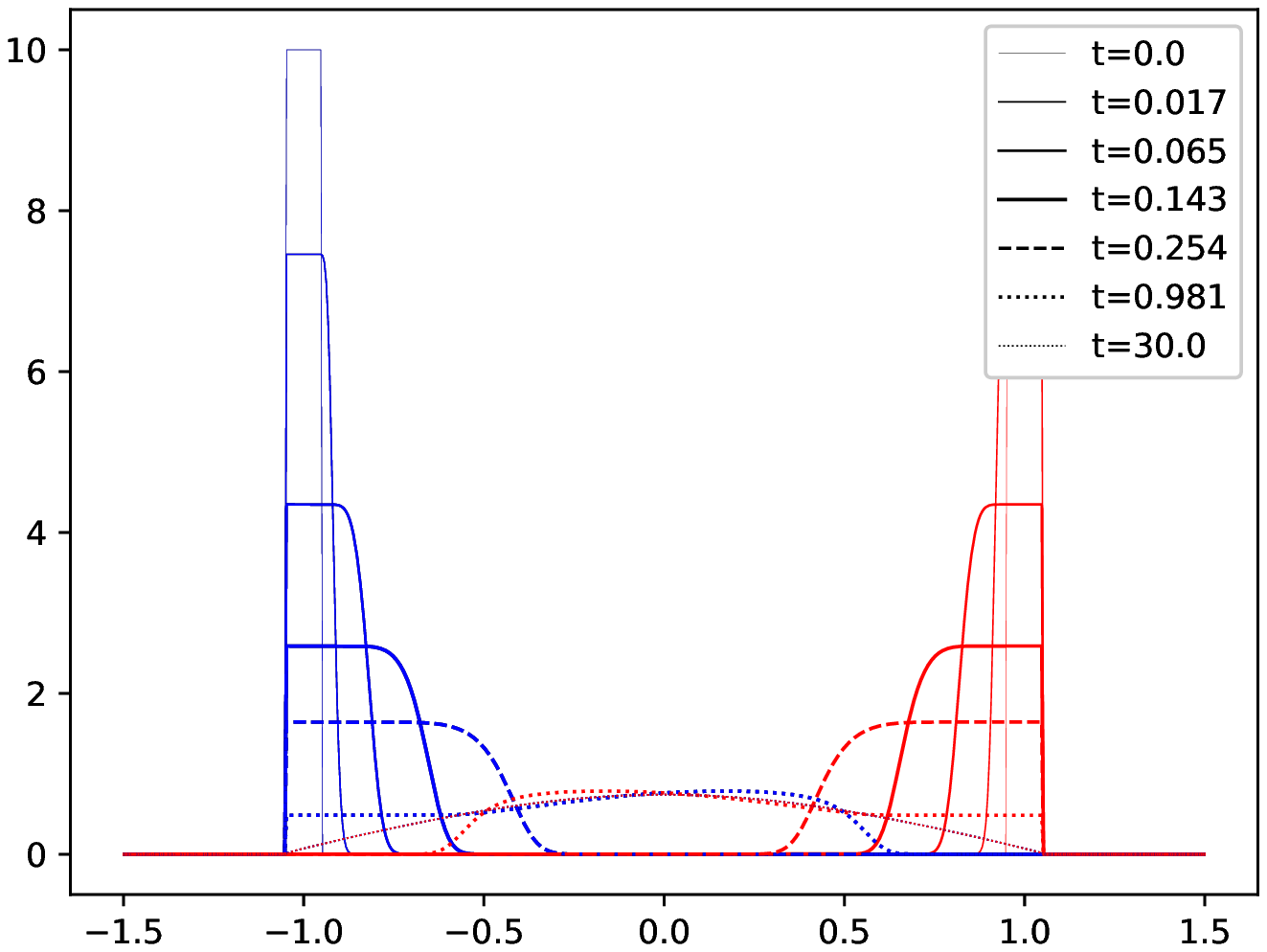}
    }
    \subfigure[Energy dissipation.]{
    	\includegraphics[width=0.30\textwidth]{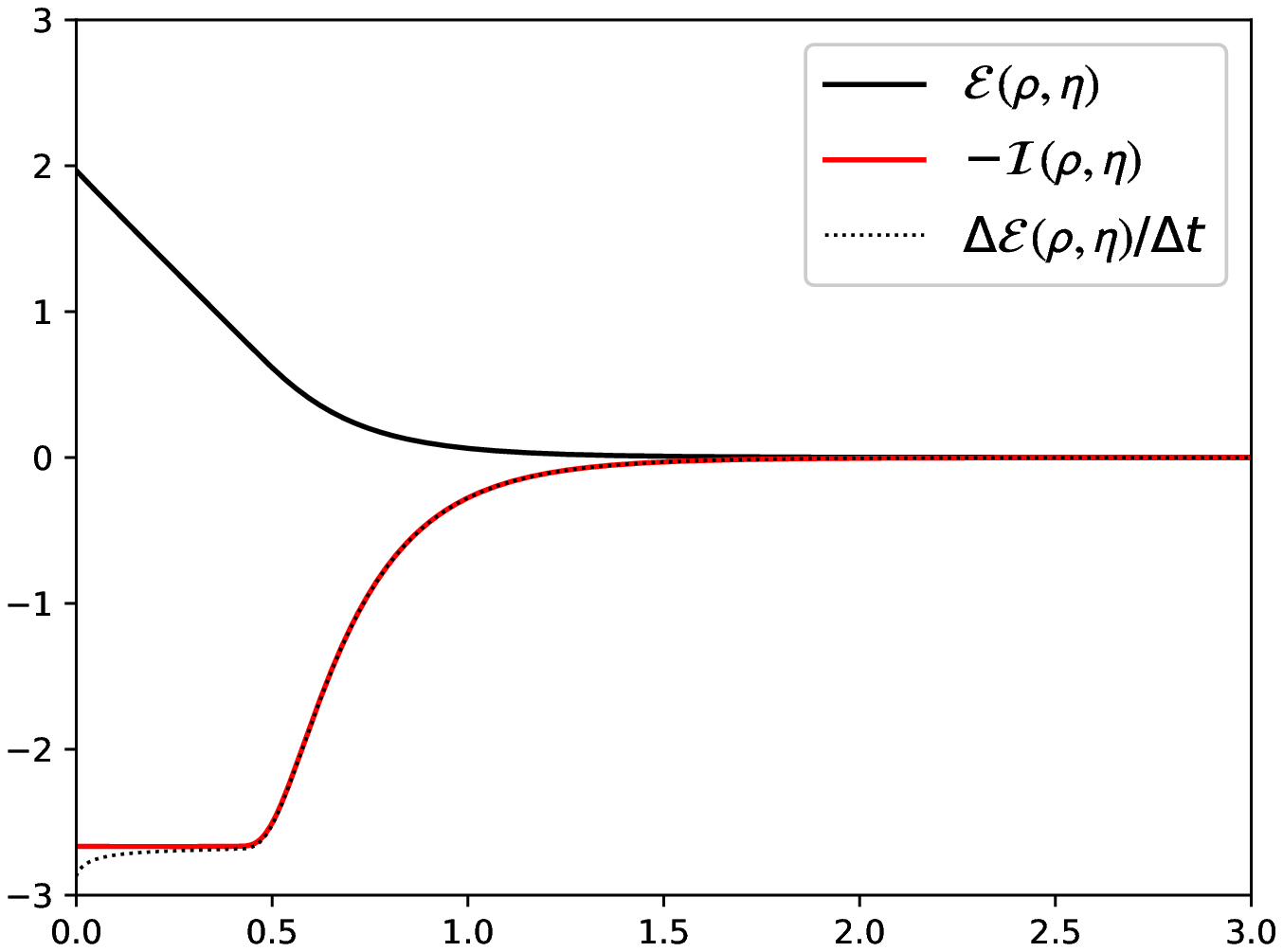}
    }
    \subfigure[Steady state.]{
    	\includegraphics[width=0.30\textwidth]{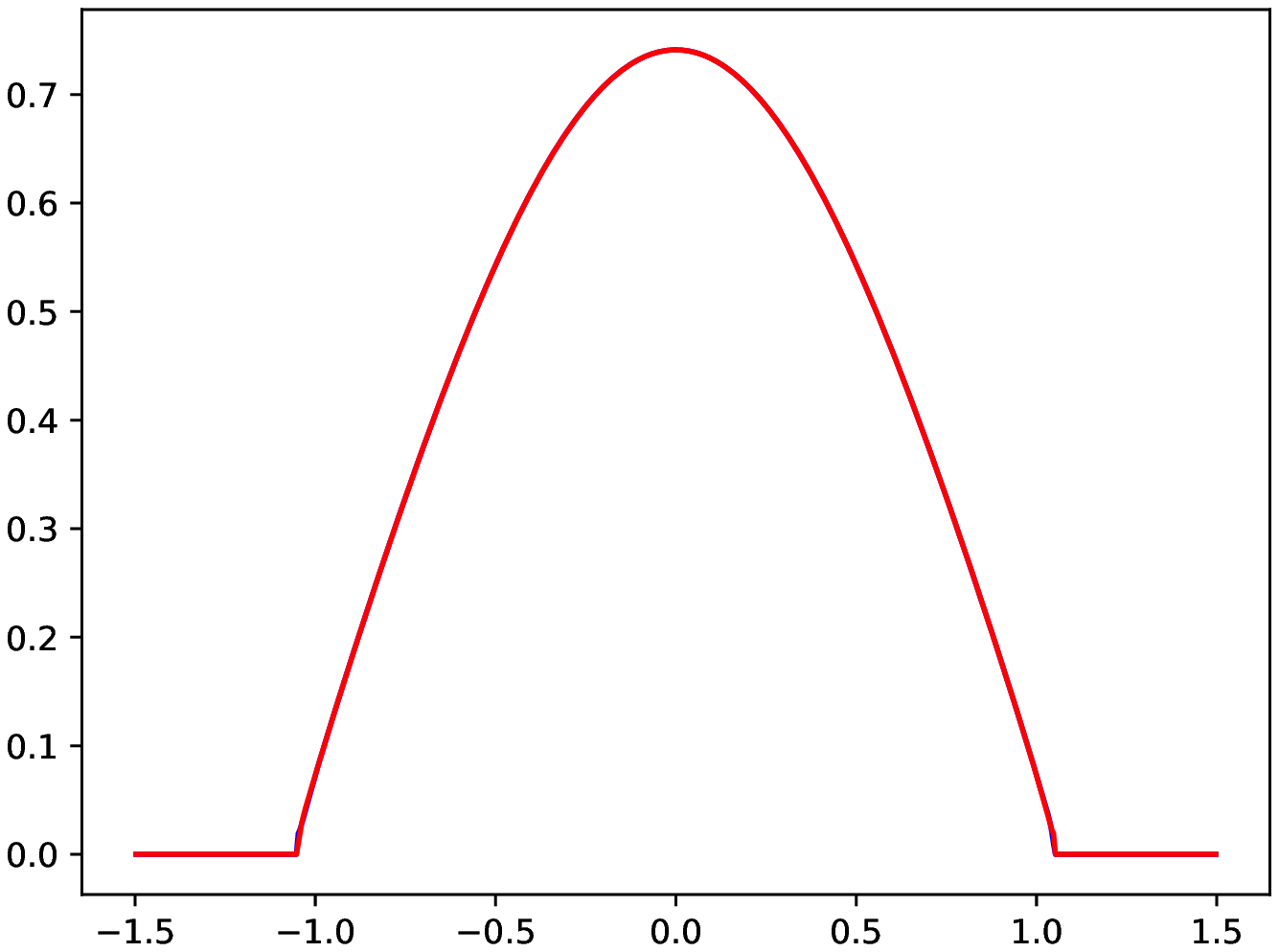}
    }
    \caption{This example has two separated indicator functions as initial data. In the left graph we see the evolution of system \eqref{eq:full_system} to the stationary state (right graph). In the middle we see the energy (black) of the solution and its dissipation (red). The dotted line is the numerical time derivative of the energy. It matches well with the analytically obtained dissipation. }
\end{figure}
\begin{figure}[ht!]
	\centering
    \subfigure[Evolution.]{
    	\includegraphics[width=0.30\textwidth]{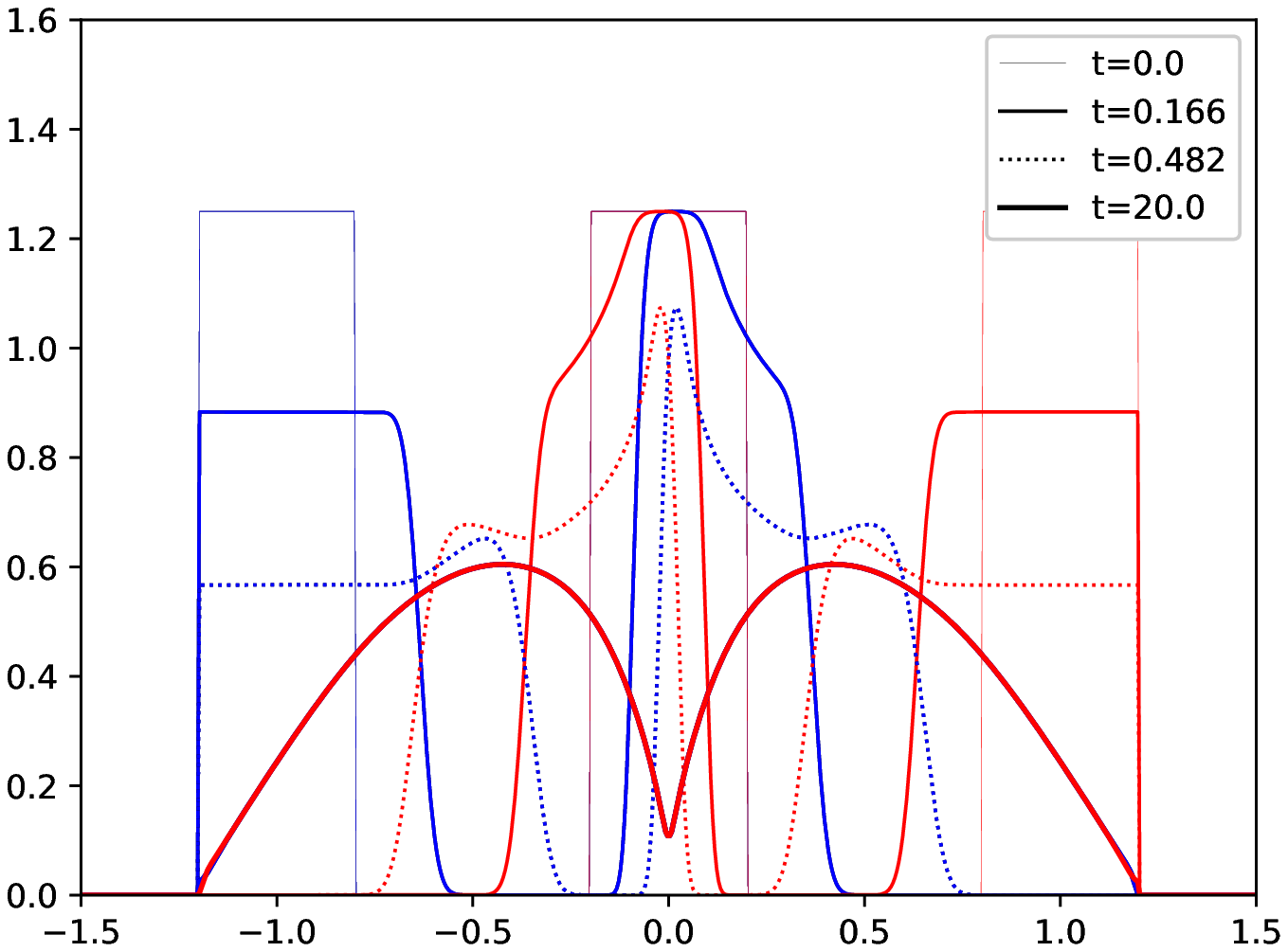}
    }
    \subfigure[Energy dissipation.]{
    	\includegraphics[width=0.30\textwidth]{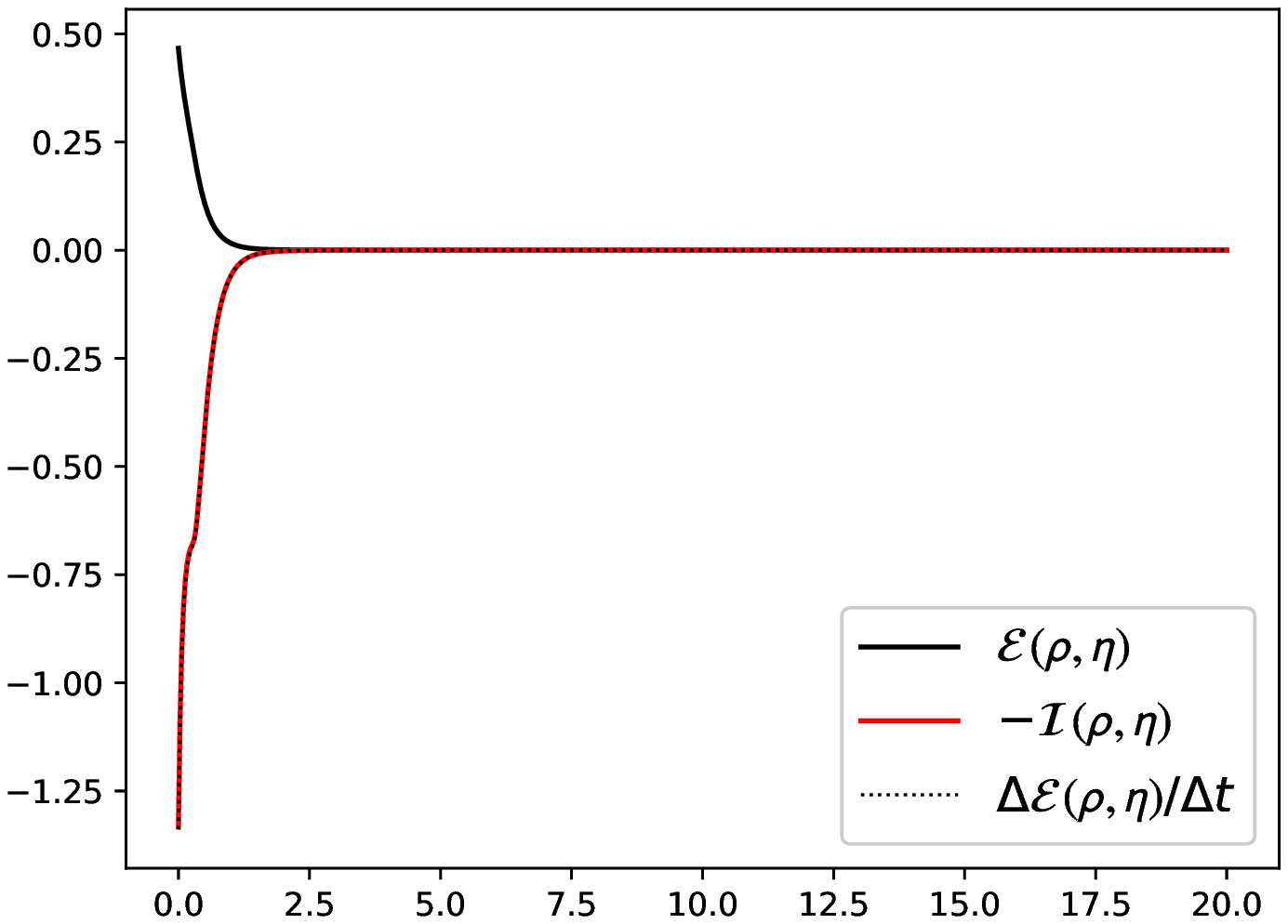}
    }
    \subfigure[Steady state.]{
    	\includegraphics[width=0.30\textwidth]{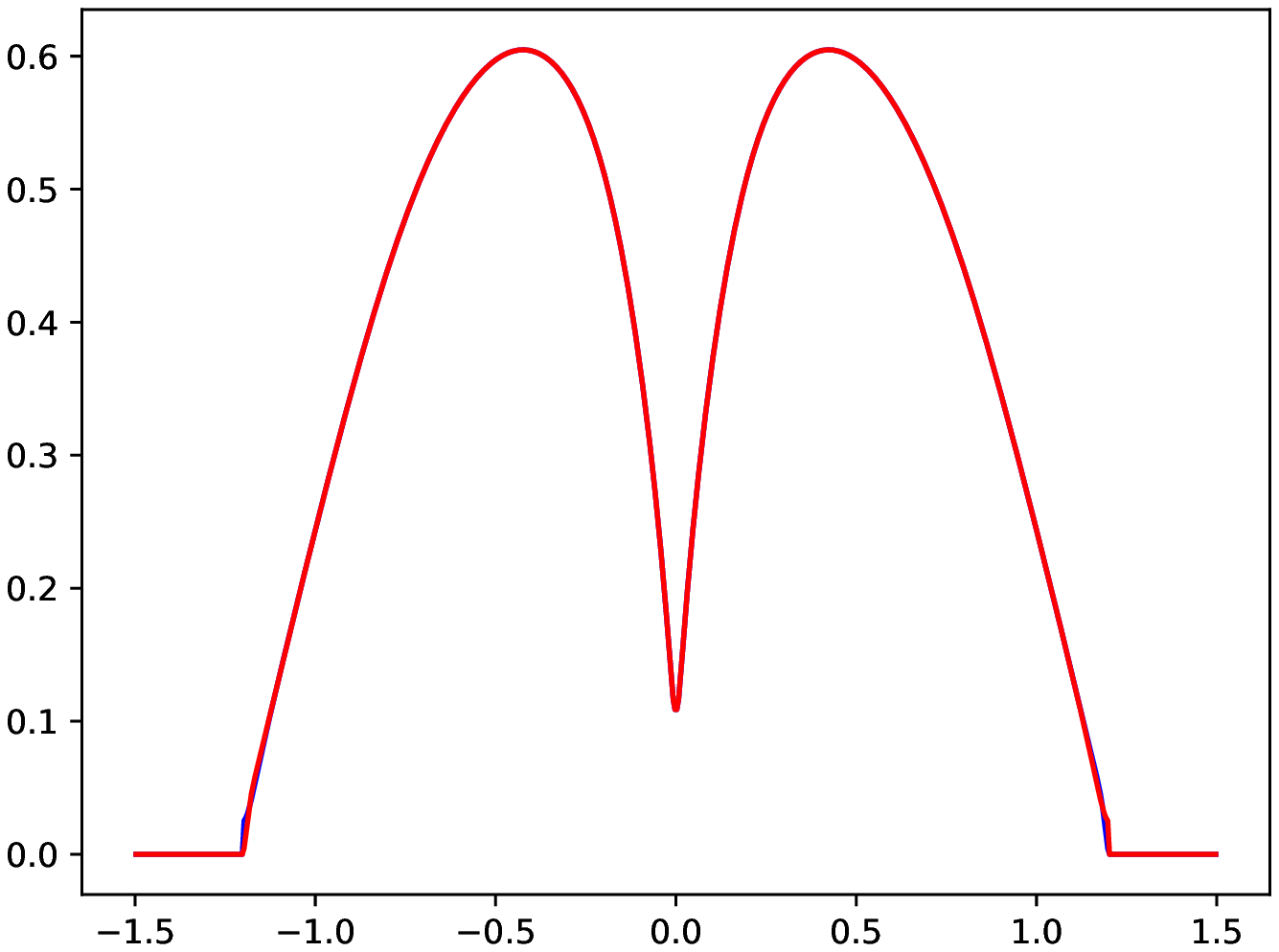}
    }
    \caption{We choose partially overlapping initial data and observe, as before that mixing occurs. The graph on the left displays the evolution of both densities at different time instances, while the rightmost graph displays the stationary state with identical densities. The graph in the middle shows the energy decay along the solution and the numerical dissipation and the analytical dissipation agree well.}
\end{figure}

\begin{remark}\label{rem:different-masses-steady-states}
We specify that in case $\rho$ and $\eta$ have masses $M_\rho\neq M_\eta$ the results in this section are no longer valid as the energy $\mathcal{F}$ is no longer bounded from below. In fact, let us assume for instance $\eta=\beta\rho$  which implies $M_\eta=\beta M_\rho$. Hence the energy becomes
\begin{align*}
\mathcal{F}(\rho,\eta)=-\frac{1}{2}M_\rho(\beta^2-1)^2\int_\R N\star\rho\,d\rho.
\end{align*}
By a simple rescaling argument the energy is shown to be unbounded from below.
\end{remark}

\section{Initial Data with Atomic Part, Non-Uniqueness \& link with Hyperbolic Systems}\label{sec:atomic}
In this section we study the evolution of measure valued initial data consisting of atomic parts. We will consider two peculiar examples: first we handle the case of two distinct Dirac deltas as initial condition, namely $\rho_0=\delta_{-1}$ and $\eta_0=\delta_1$, and then we will consider the case $ \rho_0=\delta_0$ and $\eta_0=m \delta_0 + (1-m)\delta_1$, for some $m\geq0$. In both cases we will provide two candidate weak solutions for system \eqref{eq:first} and we will show that only one of them can be selected in the spirit of Definition \ref{def:solutionmeasures}. As the notion of gradient flow in measure valued solution setting is essentially formulated in the pseudo-inverse formalism through Definition \ref{def:l2gradflow}, in this section we shall work directly with the pseudo-inverse variables.

We start by providing an explicit expression for elements in the sub-differential of $\mathfrak{F}$ in both the examples we shall consider.
\begin{proposition}\label{prop:overlapping-deltas-func+sub}
Let $m\in [0,1]$. Let $(X,Y)\in\mathcal{C}\times\mathcal{C}$ be such that $X=Y$ on $[0,m)$ and
\[\sup_{z\in [m,1]}X(z)<\inf_{z\in [m,1]}Y(z) .\] Then,
\begin{align*}
K(X,Y)=&\int_0^1[(2z-1)\mathbf{1}_{[0,m)}(z)+(2m-1)\mathbf{1}_{[m,1]}(z)]X(z)\,\d z\\&+\int_0^1[(2z-1)\mathbf{1}_{[0,m)}(z)+\mathbf{1}_{[m,1]}(z)]Y(z)\,\d z,
\end{align*}
and in particular the functional becomes
\begin{equation}\label{eq:overlapping-deltas-func}
\begin{split}
\mathfrak{F}(X,Y)=&\int_0^1[1-2z+(2z-1)\mathbf{1}_{[0,m)}(z)+(2m-1)\mathbf{1}_{[m,1]}(z)]X(z)\,\d z\\&+\int_0^1[1-2z+(2z-1)\mathbf{1}_{[0,m)}(z)+\mathbf{1}_{[m,1]}(z)]Y(z)\,\d z,
\end{split}
\end{equation}
In addition, let $X,Y\in L^2(0,1)$. Then, $\partial\bar{\mathfrak{F}}[(X,Y)]\neq\emptyset$ if and only if $(X,Y)\in\mathcal{C}\times\mathcal{C}$. In that case, if $(f_1,f_2)\in\partial\bar{\mathfrak{F}}[(X,Y)]$ then
\begin{equation*}
f_1(z)=\begin{cases}
        0,\quad & z\in[0,m),\\
        2(m-z),\quad & z\in[m,1],
		\end{cases}
 \qquad \text{and}\qquad \
f_2(z)=\begin{cases}
        0,\quad & z\in[0,m),\\
        2-2z,\quad & z\in[m,1],
		\end{cases}.
\end{equation*}
\end{proposition}
\begin{proof}
Letting $(X,Y)\in\mathcal{C}\times\mathcal{C}$ as in the statement we have $\mathcal{I}_{\mathcal{C}}(X)=\mathcal{I}_{\mathcal{C}}(Y)=0$ and
\begin{equation*}
\begin{split}
K(X,Y)=&\int_0^1\int_0^1|Y(\xi)-X(z)|\,\d z\,d\xi\\=&\int_0^m\int_0^m|Y(\xi)-X(z)|\,\d z\,\d\xi+\int_0^m\int_m^1|Y(\xi)-X(z)|\,\d z\,\d \xi\\&+\int_m^1\int_0^m|Y(\xi)-X(z)|\,\d z\,\d\xi+\int_m^1\int_m^1|Y(\xi)-X(z)|\,\d z\,\d \xi\\=:&I_1+I_2+I_3+I_4.
\end{split}
\end{equation*}
First, let us compute $I_1$ taking into account that $X=Y$ on $[0,m)$:
\begin{equation*}
\begin{split}
I_1=&\int_0^m\int_0^m|Y(\xi)-X(z)|\,\d z\,\d\xi\\=&\int_0^m\int_0^m|X(\xi)-X(z)|\,\d z\,\d\xi\\=&\int\int_{[0,m]^2\cap \{X(\xi)\ge X(z)\}}(X(\xi)-X(z))\,\d z\,\d\xi-\int\int_{[0,m]^2\cap\{X(\xi)\le X(z)\}}(X(\xi)-X(z))\,\d z\,\d\xi\\=&2\int\int_{\{X(\xi)\ge X(z)\}}(X(\xi)-X(z))\,\d z\,\d \xi.
\end{split}
\end{equation*}
Since $X$ is non-decreasing, we have
$$
\{X(\xi)\ge X(z)\}=\{\xi\ge z\}\cup\{\xi\le z\le S(\xi)\},\quad\textit{where}\quad S(\xi)=\sup\{z\in[0,1]\,|\, X(z)=X(\xi)\},
$$
and therefore we get
\begin{equation*}
\begin{split}
I_1=&2\int\int_{[0,m]^2\cap\{X(\xi)\ge X(z)\}}(X(\xi)-X(z))\,\d z\,\d\xi\\=&2\int\int_{[0,m]^2\cap\{\xi\ge z\}}(X(\xi)-X(z))\,\d z\,\d \xi\\=&\int_0^m(2z-m)(X(z)+Y(z))\,\d z.
\end{split}
\end{equation*}
Concerning the other integrals, we easily obtain
\begin{equation*}
\begin{split}
I_2=&m\int_m^1Y(z)\,\d z-(1-m)\int_0^mX(z)\,\d z,\\
I_3=&m\int_m^1X(z)\,\d z-(1-m)\int_0^mY(z)\,\d z,\\
I_4=&(1-m)\int_m^1(Y(z)-X(z))\,\d z.
\end{split}
\end{equation*}
Summing up all the contributions we have
\begin{align*}
K(X,Y)=&\int_0^1[(2z-1)\mathbf{1}_{[0,m)}(z)+(2m-1)\mathbf{1}_{[m,1]}(z)]X(z)\,\d z\\&+\int_0^1[(2z-1)\mathbf{1}_{[0,m)}(z)+\mathbf{1}_{[m,1]}(z)]Y(z)\,\d z.
\end{align*}
Then \eqref{eq:overlapping-deltas-func} follows as a direct consequence. Now, let us characterise the sub-differential of $\bar{\mathfrak{F}}$. Assume without loss of generality that $X\not\in\mathcal{C}$ and $Y\in\mathcal{C}$, which implies $\mathcal{I}_{\mathcal{C}}(X)=+\infty$ and $\mathcal{I}_{\mathcal{C}}(Y)=0$. If $(\tilde{X_1},\tilde{X_2})\in\partial\bar{\mathfrak{F}}[(X,Y)]$ we would have
\begin{equation}
	\begin{split}
		&\mathfrak{F}(R_1,R_2) + \mathcal{I}_{\mathcal{C}}(R_1)+\mathcal{I}_{\mathcal{C}} (R_2) - \mathfrak{F}(X,Y)\\
        &-\int_0^1\tilde{X_1}(z)(R_1(z)-X(z))+\tilde{X_2}(z)(R_2(z)-Y(z))\, \d z + o(\|(X,Y)-(R_1,R_2)\|) \ge\mathcal{I}_{\mathcal{C}}(X),
	\end{split}
\end{equation}
for all $(R_1,R_2)\in L^2(0,1)^2$ with $\|(X,Y)-(R_1,R_2)\|\rightarrow 0$, and in particular for all $(R_1,R_2)\in\mathcal{C}\times\mathcal{C}$. In the latter case we obviously have a contradiction because the left-hand side is finite while the right-hand side is infinite and therefore $\partial\bar{\mathfrak{F}}[(X,Y)]=\emptyset$. Let $(X,Y)\in\mathcal{C}\times\mathcal{C}$ and $(R_1,R_2)\in L^2(0,1)^2$. Now we have to consider two cases:
\begin{enumerate}
\item $(R_1,R_2)\not\in\mathcal{C}\times\mathcal{C}$;
\item $(R_1,R_2)\in\mathcal{C}\times\mathcal{C}$.
\end{enumerate}

In the first case the definition of sub-differential is trivially satisfied, whereas in the second one we get
\begin{equation*}
\begin{split}
\mathfrak{F}(R_1,R_2)-\mathfrak{F}(X,Y)=&\int_0^1[1-2z+(2z-1)\mathbf{1}_{[0,m)}(z)+(2m-1)\mathbf{1}_{[m,1]}(z)]\left[R_1(z)-X(z)\right]\,\d z\\&+\int_0^1[1-2z+(2z-1)\mathbf{1}_{[0,m)}(z)+\mathbf{1}_{[m,1]}(z)]\left[R_2(z)-Y(z)\right]\,\d z,
\end{split}
\end{equation*}
which concludes the proof.
\end{proof}
We remark in particular that under the assumptions in the previous proposition the sub-differential of $\mathfrak{F}$ is \emph{single-valued} on $\mathcal{C}\times \mathcal{C}$. Note that the case $\sup X<\inf Y$ ($m=0$) is included in the previous proposition, and  the functional becomes
\begin{equation}\label{eq:functionalpartcase}
\mathfrak{F}(X,Y)=2\int_0^1(1-z)Y(z)\,\d z-2\int_0^1zX(z)\,\d z.
\end{equation}

\subsection{The case of two distinct deltas as initial condition} Let us consider the first example, with
\begin{equation}\label{eq:two_deltas}
\rho_0=\delta_{-1}\, \qquad\mbox{ and }\qquad \eta_0=\delta_1.
\end{equation}
At the level of weak (measure) solutions, both
 \begin{align}
	\label{eq:two_diracs_explicit_soln}
	\rho(t,x) = \frac{1}{2t}\mathbf{1}_{[-1,-1+2t]}(x), \qquad \mbox{and} \qquad \eta(t,x) = \frac{1}{2t}\mathbf{1}_{[1-2t,1]},
\end{align}
and $(\tilde \rho, \tilde\eta)$ given by
\begin{align}
	\tilde \rho(t,x)= \delta_{t-1}, \quad\mbox{and}\quad \tilde \eta(t,x)=\delta_{1-t},
\end{align}
	satisfy system \eqref{eq:full_system} in the weak sense on $[0,1/2]\times\R$ and equal \eqref{eq:two_deltas} at $t=0$, see considerations in Subsection \ref{sssec:ws1}. The corresponding pseudo-inverse functions $(X,Y)$, given by
 \begin{align}\label{eq:ps-inv-good}
	X(t,z) =-1+2zt,\qquad\mbox{and}\qquad Y(t,z) = 1+t(2z-2),
\end{align}
as well as $(\tilde X, \tilde Y)$, given by
	\begin{align}\label{eq:ps-inv-bad}
		\tilde X(t,z)=-1+t,\qquad\mbox{and}\qquad \tilde Y(t,z)=1-t,
	\end{align}
satisfy system \eqref{eq:system-pseudo-inverse} in the strong sense on $[0,1/2]\times[0,1]$, see Subsection \ref{sssec:ss1}. Working in the context of pseudo-inverses we can show, in Subsection \ref{sssec:csd1} that actually only the time derivative of the pseudo-inverse in \eqref{eq:ps-inv-good} is an element of the sub-differential.

\subsubsection{Weak measure solutions}\label{sssec:ws1}
In order to prove that $(\rho,\eta)$ in \eqref{eq:two_diracs_explicit_soln} is a weak solution of the system we begin by simplifying the velocity term. To this end we compute the convolution with the Heaviside function
\begin{align*}
	(N'\star \rho)(t,x) &= \int_{-\infty}^{+\infty} \sign(x-y)\rho(t,y)\d y\\
    &= \left\{
    \begin{array}{cl}
    	-1, & \mbox{if } x \leq -1,\\
        \dfrac{x+1-t}{t}, & \mbox{if } x\in(-1,-1+2t],\\
        1, &\mbox{else.}
    \end{array}
    \right.
\end{align*}
Similarly, for the convolution with the second species we obtain
\begin{align*}
	(N'\star \eta)(t,x) &= \int_{-\infty}^{+\infty} \sign(x-y)\eta(t,y)\d y\\
    &= \left\{
    \begin{array}{cl}
    	-1, & \mbox{if } x \leq 1-2t,\\
        \dfrac{x-1+t}{t}, & \mbox{if } x\in(1-2t,1],\\
        1, &\mbox{else.}
    \end{array}
    \right.
\end{align*}
We claim that $\rho,\eta$ as defined above are weak solutions to system \eqref{eq:full_system}. Here we only check that $\rho$ is a weak solution as the computation for the second species is done in an analogous way. Now, let $\phi\in C_c^\infty$ and consider the weak formulation
\begin{align*}
	&\int_0^{1/2}\int_{-\infty}^\infty \rho \bigg[\phi_t -N'\star(\eta-\rho)\phi_x\bigg]\d x\d t =: I_t + I_x,
\end{align*}
where
\begin{align*}
	I_t := \int_0^{1/2}\int_{-\infty}^\infty \rho \phi_t(t,x) \d x \d t,
\end{align*}
and
\begin{align*}
	I_x := - \int_0^{1/2}\int_{-\infty}^\infty \rho N'\star(\eta-\rho)\phi_x(t,x)\d x\d t.
\end{align*}
Let us begin by simplifying the time related term. By changing the order of integration it is easy to see that
\begin{align*}
 	I_t &= \int_0^{1/2}\int_{-\infty}^\infty \rho \phi_t(t,x) \d x \d t = \int_{0}^{1/2} \int_{-1}^{-1+2t} \frac{1}{2t} \phi_t(t,x) \d t\d x\\
    &= \int_{-1}^0 \int_{\frac{x+1}{2}}^{1/2} \frac{1}{2t^2} \phi(t,x) \d t\d x + \int_{-1}^0 \left[\frac{1}{2t}\phi(t,x)\right]_{t=\frac{x+1}{2}}^{1/2}\d x,
\end{align*}
by an integration by parts. Hence, switching the order of integration another time and simplifying the boundary term we obtain
\begin{align*}
	I_t =\int_{0}^{1/2}\int_{-1}^{-1+2t}\frac{1}{2t^2}\phi(t,x) \d x \d t - \int_{-1}^0 \frac{1}{x+1}\phi\left(\frac{x+1}{2},x\right)\d x.
\end{align*}
A change of variables $x+1 = 2t$ finally yields
\begin{align}
	I_t = \int_{0}^{1/2}\int_{-1}^{-1+2t}\frac{1}{2t^2}\phi(t,x) \d x \d t -\int_{0}^{1/2} \frac{1}{t}\phi\left(t,2t-1\right)\d t.
\end{align}

Next we shall address the term space related term. We observe
\begin{align*}
	I_x &= - \int_0^{1/2}\int_{-1}^{-1+2t} \frac{1}{2t} \left[-1-\frac{x+1-t}{t}\right]\phi_x(t,x)\d x\d t\\
    &= \int_0^{1/2}\int_{-1}^{-1+2t}  \frac{x+1}{2t^2}\phi_x(t,x)\d x\d t.
\end{align*}
An integration by parts yields
\begin{align*}
	I_x &=\int_0^{1/2}\int_{-1}^{-1+2t}  \frac{x+1}{2t^2}\phi_x(t,x)\d x\d t\\
    &= -\int_0^{1/2}\int_{-1}^{-1+2t} \frac{1}{2t^2} \phi(t,x)\d x\d t + \int_0^{1/2}\left[\frac{x+1}{2t^2}\phi(t,x)\right]_{x=-1}^{-1+2t}\d t\\
    &= -\int_0^{1/2}\int_{-1}^{-1+2t} \frac{1}{2t^2} \phi(t,x)\d x\d t + \int_0^{1/2}\frac{1}{t}\phi(t,2t-1)\d t.
\end{align*}
Upon adding up $I_t$ and $I_x$ we observe
\begin{align*}
	I_t + I_x = 0,
\end{align*}
i.e. $\rho$ is a weak solution of the first equation in system \eqref{eq:full_system}. Similarly it can be shown that $\eta$ is a weak solution to the second equation in \eqref{eq:full_system}.

Next we show that $(\tilde \rho, \tilde \eta)$ is also a weak solution. As before we compute the terms including the convolutions first. It is easy to check that
$$
(N'\star\tilde{\rho})(x)=\sign(x-t+1),\qquad\text{and}\qquad(N'\star\tilde{\eta})(x)=\sign(x+t-1),
$$
for all $x\in\R$, thus the velocity is given by $u:=N'\star(\tilde{\eta}-\tilde{\rho})=\sign(x+t-1)-\sign(x-t+1)$. Let us consider a test function $\phi\in C_c^\infty$ in order to check the weak formulation as follows:
\begin{align*}
\int_0^1\int_\R\tilde{\rho}[\phi_t-u\phi_x]\,\d x\,\d t&=\int_0^1\phi_t(t,t-1)-u(t,t-1)\phi_x(t,t-1)\,\d t\\&=\int_0^1\phi_t(t,t-1)+\phi_x(t,t-1)\d t\\
	&=\int_0^1\frac{\d}{\d t}\left(\phi(t,t-1)\right)\,\d t\\&=0.
\end{align*}
Arguing similarly for $\tilde{\eta}$ we have $(\tilde{\rho},\tilde{\eta})$ is a weak solution to system \eqref{eq:full_system} with initial data $\rho_0=\delta_{-1}$ and $\eta_0=\delta_{1}$ as well.

\subsubsection{Strong solutions in the pseudo-inverse formalism}\label{sssec:ss1}
Next, let us show that $(X,Y)$ defined in \eqref{eq:ps-inv-good} is the solution to system \eqref{eq:system-pseudo-inverse} in the strong sense. Using that $t<1/2$ there holds
\begin{align*}
	&\int_0^1 \sign(X(t,z)-X(t,\xi)) \d \xi -\int_0^1 \sign(X(t,z)-Y(t,\xi)\d \xi\\
	&=\int_0^1\sign(2t(z-\xi))\d\xi -\int_0^1 \sign(-2+2t(z-\xi+1)) \d\xi\\
	&=\int_0^z \d\xi -\int_z^1\d\xi +1 \\
	&= 2z\\
	&= \frac{\partial}{\partial t} X(t,z),
\end{align*}
and
\begin{align*}
	&\int_0^1 \sign(Y(t,z)-Y(t,\xi)) \d \xi -\int_0^1 \sign(Y(t,z)-X(t,\xi)\d \xi\\
	&=\int_0^1\sign(2t(z - \xi))\d\xi -\int_0^1 \sign(2+2t(z-\xi-1)) \d\xi\\
	&=\int_0^z\d\xi -\int_z^1\d\xi -1 \\
	&= 2z-2\\
	&= \frac{\partial}{\partial t} Y(t,z).
\end{align*}
As for the second pair of pseudo–inverses \eqref{eq:ps-inv-bad} we observe
\begin{align*}
& \int_0^1 \sign(\tilde X(t,z)- \tilde X(t,\xi)) \d \xi -\int_0^1 \sign(\tilde X(t,z)-\tilde Y(t,\xi)\d \xi\\
	&=\int_0^1 \sign(0)\d \xi -\int_0^1 \sign(2(t-1))\d \xi\\
	&=1\\
	&=\frac{\partial}{\partial t} \tilde X(t,z),
\end{align*}
and similarly the equation for $\tilde Y$ is satisfied.

\subsubsection{Characterisation of the sub-differential}\label{sssec:csd1} We need to check the differential inclusion. According to Proposition \ref{prop:overlapping-deltas-func+sub} we have
$$
\frac{\partial}{\partial t}\begin{pmatrix}
X(t,z)\\
Y(t,z)
\end{pmatrix}=\begin{pmatrix}
2z\\
2z-2
\end{pmatrix}\in-\partial\bar{\mathfrak{F}}\big[\big(X(t,z),Y(t,z)\big)\big],
$$
as we claimed. However, there holds $\frac{\partial}{\partial t} \tilde X(t,z)=1\neq2z$, which shows the pair $(\tilde X, \tilde Y)$ defined in \eqref{eq:ps-inv-bad} is not a solution to system \eqref{eq:system-pseudo-inverse} in the sense of Definition \ref{def:l2gradflow} for the given initial data.

\begin{remark}\label{rem:example1}
In this example, both species are initially concentrated at one point, but there is no overlap between them. Therefore, only the intraspecific energies are affected at ``singular points'', i.e. at points in which the convolution kernel is not smooth. In fact, the interspecific energy is not effected by the Lipschitz point at the origin. In this sense, one expects the qualitative behaviour of this system to be essentially the same as in the one species case, see \cite{BCDFP}. More precisely, both species get immediately absolutely continuous w.r.t. to the Lebesgue measure. The attractive cross-interaction energy makes the two patches get closer to each other until they eventually merge. In conclusion, the existence of two distinct measure solutions in this example is not a distinctive feature of the two species system, but rather an extension of a property holding in the one species case.
\end{remark}

\subsection{The case of two overlapping deltas as initial condition}
 Let $0\leq m < 1$ be given and initialise system \eqref{eq:full_system} as follows
 \begin{equation}\label{eq:two_over}
  \rho_0=\delta_0,\quad\mbox{and}\quad \eta_0=m \delta_0 + (1-m)\delta_1.
 \end{equation}
 Then the pair $(\rho,\eta)$ given by
 \begin{align*}
  \rho(t,x)= m\delta_0 + \frac{1}{2t}\mathbf{1}_{[0,2(1-m)t]}(x),\quad\mbox\quad \eta(t,x)=m\delta_0 + \frac{1}{2t}\mathbf{1}_{[1-2(1-m)t,1]}(x),
 \end{align*}
 is a weak solution to system \eqref{eq:full_system} on $[0,T)\times[0,1]$, with $T:=\frac{1}{4(1-m)}$, as well as the pair $(\tilde \rho,\tilde \eta)$ given by
 \begin{align*}
  \tilde \rho(t,x)=m\delta_0 + (1-m)\delta_{(1-m)t}, \quad\mbox{and}\quad \tilde \eta(t,x)=m\delta_0 + (1-m)\delta_{1-(1-m)t}.
 \end{align*}
Moreover, for $t\in[0,T)$, the associated pseudo-inverse functions
\begin{align*}
	X(t,z) = 2t (z-m) \mathbf{1}_{[m,1]}(z), \qquad Y(t,z) = \big(1-2t (1-z)\big) \mathbf{1}_{[m,1]}(z),
\end{align*}
and
\begin{align*}
 \tilde X(t,z) = (1-m)t \mathbf{1}_{[m,1]}(z), \qquad \tilde Y(t,z) = \big(1-(1-m)t\big) \mathbf{1}_{[m,1]}(z).
\end{align*}
are both strong strong solutions to system \eqref{eq:system-pseudo-inverse}, but only (X,Y) is the gradient flow solution in the sense of Definition \ref{def:l2gradflow}.
\subsubsection{Weak measure solutions}\label{sssec:ws2} Let us start by verifying that $(\rho,\eta)$ is a weak solution to system \eqref{eq:full_system} on $[0,T)\times[0,1]$. Next we compute the vector field for $\rho$ on $x\in[0,1/2]$. For the self-interactions we get
\begin{align*}
	(N'\star\rho)(t,x) &= \int_\R \sign(x-y)\rho(t,y)\d y \\
    &= m\, \sign(x) + \frac{1}{2t} \int_0^{2(1-m)t}\sign(x-y)\d y\\
    &=m\, \sign(x) + \frac{2x-2(1-m)t}{2t}\\
    &=m\, \sign(x) + \frac{x}{t}  +m-1,
\end{align*}
whereas, for the cross-interactions, we get
\begin{align*}
	(N'\star\eta)(t,x) &= m\,\sign(x) - \int_{1-2(1-m)t}^1 \frac{1}{2t}\d y \\
    &=m\,\sign(x) + m - 1.
\end{align*}
Hence, the velocity on $[0,1/2]$ is given by
\begin{align*}
	u = N'\star(\eta-\rho) = -\frac{x}{t}.
\end{align*}
We shall now verify that $\rho$ and $\eta$  are weak solutions. It is easy to see that
\begin{align*}
	I&=\int_0^T\int_\R \rho(t,x)\bigg[\phi_t(t,x) - u \phi_x(t,x)\bigg]\d x \d t\\
    &=m\int_0^T \phi_t(t,0) - u(t,0)\phi_x(t,0)\d t\\
    &\quad + \underbrace{\int_0^T\int_0^{2(1-m)t} \frac{1}{2t} \phi_t(t,x)\d x \d t}_{=:I_1} +\underbrace{ \int_0^T\int_0^{2(1-m)t}\frac{x}{2t^2}\phi_x(t,x)\d x\d t}_{=:I_2}.
\end{align*}
We will treat each term individually. Note that $u(t,0)=0$ for all $t>0$. Together with the fact that $\phi$ is compactly supported and the application of the fundamental theorem the first term vanishes and it remains to treat the terms $I_1, I_2$. Using Fubini's theorem and an integration by parts we may write
\begin{align*}
	I_1 &=\int_0^T\int_0^{2(1-m)t} \frac{1}{2t}\phi_t(t,x)\d x \d t \\
    &=\int_0^{2(1-m)T}\int_{x/(2(1-m))}^T \frac{1}{2t}\phi_t(t,x)\d t \d x\\
    &=\int_0^{2(1-m)T}\int_{x/(2(1-m))}^T \frac{1}{2t^2} \phi(t,x)\d t \d x - \int_0^{2(1-m)T}\frac{1-m}{x}\phi\left(\frac{x}{2(1-m)},x\right) \d x\\
    &=\int_0^{2(1-m)T}\int_{x/(2(1-m))}^T \frac{1}{2t^2} \phi(t,x)\d t \d x -
    \int_0^{T}\frac{1-m}{t}\phi(t,2(1-m)t)\d t.
\end{align*}
As for the second term a simple integration by parts yields
\begin{align*}
	I_2 &=\int_0^T\int_0^{2(1-m)t}\frac{x}{2t^2} \phi_x(t,x)\d x\d t\\
    &= - \int_0^T\int_0^{2(1-m)t}\frac{1}{2t^2}\phi(t,x)\d x \d t + \int_0^T \frac{1-m}{t} \phi(t,2(1-m)t)\d t.
\end{align*}
Thus we get
\begin{align*}
	I &= I_1 + I_2=0.
\end{align*}
Now, we need to check that $(\tilde \rho, \tilde\eta)$ satisfies the weak formulation. Here we only check the statement for $\tilde\rho$ as the second species is shown analogously. Again we compute the velocity field for $\tilde \rho$.
 \begin{align*}
  u(t,x)&=N'\star(\tilde \eta - \tilde \rho)(t,x)\\
   &= (1-m)\big(\sign(x-1+(1-m)t)-\sign(x-(1-m)t)\big).
 \end{align*}
 Note that $u(t,0) = 0$ and $u(t,(1-m)t)= -(1-m)$. Thus there holds
 \begin{align*}
  &\int_0^T\int_\R \tilde \rho \bigg[\phi_t(t,x) - u(t,x) \phi_x(t,x)\bigg]\d x \d t\\
  &=m\int_0^T \phi_t(t,0) + 0\phi_x(t,0)\d t\\
  &\quad + (1-m)\int_0^T \phi_t(t,(1-m)t) + (1-m) \phi_x(t,(1-m)t)\d t\\
  &=\int_0^T m\ddt \phi(t,0) + (1-m)\ddt \phi(t,(1-m)t) \d t\\
  &=0,
 \end{align*}
 by the fundamental theorem of calculus.

\subsubsection{Strong solutions in the pseudo-inverse formalism}\label{sssec:ss2}Let us now consider the associated pseudo-inverse functions $X$ and $Y$, given by
\begin{align*}
	X(t,z) = 2t (z-m) \mathbf{1}_{[m,1]}(z), \quad\mbox{and}\quad Y(t,z) = \big(1-2t (1-z)\big) \mathbf{1}_{[m,1]}(z),
\end{align*}
for $z\in[0,1]$ and $0\leq t < T$.
 $(X,Y)$ is a solution to system \eqref{eq:system-pseudo-inverse} in a strong sense, in fact
\begin{align*}
&\int_0^1 \sign(X(z) - X(\xi)) \d \xi - \int_0^1 \sign(X(z) - Y(\xi)) \d \xi\\
 &=\int_0^1 \sign\bigg(2t (z-m) \mathbf{1}_{[m,1]}(z) - 2t (\xi-m) \mathbf{1}_{[m,1]}(\xi)\bigg)\d \xi\\
 &\quad - \int_0^1 \sign\bigg(2t (z-m) \mathbf{1}_{[m,1]}(z)  - \big(1-2t (1-\xi)\big) \mathbf{1}_{[m,1]}(\xi)\bigg) \d \xi\\
&=(m-1)\mathbf{1}_{[0,m]}(z)+(2z-1)\mathbf{1}_{[m,1]}(z)+(1-m)\mathbf{1}_{[0,m]}(z)+(1-2m)\mathbf{1}_{[m,1]}(z)\\
&=2(z-m)\mathbf{1}_{[m,1]}(z)\\
&=\frac{\partial}{\partial t}X(t,z),
\end{align*}
and
\begin{align*}
&\int_0^1 \sign(Y(z) - Y(\xi)) \d \xi - \int_0^1 \sign(Y(z) - X(\xi)) \d \xi\\
 &=\int_0^1 \sign\bigg(\big(1-2t(1-z)\big)\mathbf{1}_{[m,1]}(z) - \big(1-2t(1-\xi)\big)\mathbf{1}_{[m,1]}(\xi)\bigg)\d \xi\\
 &\quad - \int_0^1 \sign\bigg(\big(1-2t(1-z)\big)\mathbf{1}_{[m,1]}(z)  - 2t(\xi-m)\mathbf{1}_{[m,1]}(\xi)\bigg) \d \xi\\
&=(m-1)\mathbf{1}_{[0,m]}(z)+(2z-1)\mathbf{1}_{[m,1]}(z)+(1-m)\mathbf{1}_{[0,m]}(z)-\mathbf{1}_{[m,1]}(z)\\
&=2(z-1)\mathbf{1}_{[m,1]}(z)\\
&=\frac{\partial}{\partial t} Y(t,z),
\end{align*}
as we claimed. Moreover, the pair $(\tilde X,\tilde Y)$ of pseudo-inverses associated to the moving Diracs, i.e.,
\begin{align*}
 \tilde X(t,z) = (1-m)t \mathbf{1}_{[m,1]}(z), \qquad \tilde Y(t,z) = \big(1-(1-m)t\big) \mathbf{1}_{[m,1]}(z),
\end{align*}
is another strong solution to system \eqref{eq:system-pseudo-inverse}, since
\begin{align*}
 &\int_0^1 \sign(X(z) - X(\xi)) \d \xi - \int_0^1 \sign(X(z) - Y(\xi)) \d \xi\\
 &=\int_0^1 \sign((1-m)t\mathbf{1}_{[m,1]}(z) - (1-m)t\mathbf{1}_{[m,1]}(\xi))\d \xi \\
 &\quad -\int_0^1 \sign((1-m)t \mathbf{1}_{[m,1]}(z) - (1-(1-m)t) \mathbf{1}_{[m,1]}(\xi)) \d \xi\\
 &= (1-m)\mathbf{1}_{[m,1]}(z)\\
 &=\frac{\partial}{\partial t} \tilde X(t,z),
\end{align*}
and, repeating the same computation for $\tilde Y(t,z)$, we have that $\frac{\partial}{\partial t}\tilde Y(t,z)=-(1-m)\mathbf{1}_{[m,1]}(z)$.

\subsubsection{Characterisation of the sub-differential}\label{sssec:csd2}
We notice that both solutions satisfy the assumptions of Proposition \ref{prop:overlapping-deltas-func+sub}. Since for $(t,z)\in[0,T)\times[0,m)$
$$
\frac{\partial}{\partial t}\begin{pmatrix}
X(t,z)\\
Y(t,z)
\end{pmatrix}=\begin{pmatrix}
0\\
0
\end{pmatrix}
$$
and, for $(t,z)\in[0,T)\times[m,1]$,
$$
\frac{\partial}{\partial t}\begin{pmatrix}
X(t,z)\\
Y(t,z)
\end{pmatrix}=\begin{pmatrix}
2(z-m)\\
2(z-1)
\end{pmatrix},
$$
we  have $
\frac{\partial}{\partial t}\begin{pmatrix}
X(t,z)\\
Y(t,z)
\end{pmatrix}\in-\partial\bar{\mathfrak{F}}\big[\big(X(t,z),Y(t,z)\big)\big],
$ so that we can affirm $(X,Y)$ is a gradient flow solution to system \eqref{eq:system-pseudo-inverse}. In conclusion, $(\tilde{X},\tilde{Y})$ is not a gradient flow solution since $\frac{\partial}{\partial t} \tilde X(t,z)=1-m\neq2(z-m)$, as we claimed.

\begin{remark}\label{rem:example2}
  Unlike the case of two separate Dirac deltas, the phenomenon arising in this example is indeed a distinctive feature of the two species case. This time the inter-specific energy is indeed affected at the singular point, since both species are present at the same position initially. The common mass $m$ at the point zero is driven both by a self-repulsion and by a cross-attraction effect annihilating each other and producing no movement at all as a result. The extra mass of $\rho$ is instead only driven by self-repulsion, and therefore it gets smoothed. At the point $1$, only the smoothing effect occurs, as there is no singular cross-interaction. There is a significant aspect in this solution: \emph{the gradient flow solution maintains a bit of its initial atomic part}, which never happens in the one species case.
\end{remark}

\subsection{Link with hyperbolic systems}\label{subsec:hyp}

In this section we want to highlight the link between system \eqref{eq:full_system} and a particular nonlinear $2\times2$ system of conservation laws in one space dimension, see \cite{Bressan,Dafermos}. Indeed, considering the cumulative distribution functions $F$ and $G$ of $\rho$ and $\eta$ respectively (as defined in \eqref{eq:distribution-functions}), we can rewrite system \eqref{eq:full_system} as
\begin{equation}\label{eq:hypsys}
\begin{cases}
	\partial_t F +2(F-G) \partial_x F=0,\\
    \partial_t G +2(G-F) \partial_x G=0,
\end{cases}
\end{equation}
or in the equivalent matrix form
\begin{align*}
	\begin{pmatrix}
	\partial_t F\\
	\partial_t G
	\end{pmatrix}+
	\left(
    \begin{array}{cc}
    2(F-G)& 0\\
    0     & 2(G-F)
    \end{array}
	\right)
    \cdot\begin{pmatrix}
	\partial_x F\\
	\partial_x G
	\end{pmatrix}=0.
\end{align*}
We stress that the initial condition $F_0, G_0$ for $F$ and $G$ are non-decreasing and achieving values in $[0,1]$, with $F_0(-\infty)=G_0(-\infty)=0$ and $F_0(+\infty)=G_0(+\infty)=1$. System \eqref{eq:hypsys} is hyperbolic, though not strictly, as the eigenvalues are
\begin{align*}
&\lambda_1(F,G)=2(F-G)\\
&\lambda_2(F,G)=2(G-F),
\end{align*}
and $\lambda_1=\lambda_2$ on the diagonal $F=G$. Moreover, system \eqref{eq:hypsys} is \emph{nonconservative}, in the sense that there exists no flux function $\bm{f}:[0,1]\times[0,1]\to\R^2$ such that \eqref{eq:hypsys} can be written as
\begin{equation}\label{eq:hypsiscons}
\partial_t U + \partial_x(\bm{f}(U))=0,\qquad U=(F,G).
\end{equation}
To our knowledge, no general theories on hyperbolic systems currently allow to define a notion of entropy solution for such a system due to the lack of strict hyperbolicity. In particular, there is no canonical way to define a suitable Riemann solver.

The link between \eqref{eq:hypsys} and \eqref{eq:full_system} can be easily established at the level of weak solutions for \eqref{eq:hypsys} with sufficiently smooth initial data, which will correspond to weak solutions for \eqref{eq:full_system}. Such a link is slightly more tricky at the level of discontinuous solutions. On the other hand, the use of the Evolution Variational Inequality for \eqref{eq:full_system} seems to be like a natural way to characterise a solution for \eqref{eq:hypsys} as well. This task will be performed in a future work. In this subsection we will just display the gradient flow solutions found in the previous subsections at the level of the hyperbolic system \eqref{eq:hypsys}, as relevant examples of solutions of Cauchy problems which can be solved via the composition of two Riemann problems.

Let us start by considering the Cauchy problem
\begin{equation}\label{eq:riemannpr}
\begin{cases}
	\partial_t F +2(F-G) \partial_x F=0\\
    \partial_t G +2(G-F) \partial_x G=0
\end{cases},\quad F_0=\begin{cases}
0\quad x<-1\\
1 \quad x\geq -1
\end{cases}, \quad
G_0=\begin{cases}
			0 \quad x<1\\
           1\quad x\geq 1
          \end{cases},
\end{equation}
which correspond to the initial condition for \eqref{eq:full_system}
\begin{align*}
  \rho_0=\delta_{-1},\quad\mbox{and}\quad \eta_0=\delta_1.
 \end{align*}
As anticipated, we construct the solution (only for short time) by solving two separate Riemann problems, i.e.
\begin{equation}\label{eq:riemannprONE}
\begin{cases}
	\partial_t F +2(F-G) \partial_x F=0\\
    \partial_t G +2(G-F) \partial_x G=0
\end{cases},
          \quad F_0=\begin{cases}
0\quad x<-1\\
1 \quad x\geq -1
\end{cases}, \quad
G_0=\begin{cases}
			0 \quad x<-1\\
           0 \quad x\geq -1
          \end{cases},
\end{equation}
and
\begin{equation}\label{eq:riemannprTWO}
\begin{cases}
	\partial_t F +2(F-G) \partial_x F=0\\
    \partial_t G +2(G-F) \partial_x G=0
\end{cases},
\quad F_0=\begin{cases}
1\quad x<1\\
1 \quad x\geq  1
\end{cases}, \quad
G_0=\begin{cases}
			0 \quad x<1\\
           1 \quad x\geq  1
          \end{cases}.
\end{equation}
On the basis of our previous results, we know the solution at the level of pseudo-inverses functions, i.e. for $z\in[0,1]$ and for $t$ small enough,
\begin{align*}
X(t,z)=2tz-1, \qquad \text{and}\qquad Y(t,z)=2t(z-1)+1.
\end{align*}
Computing the corresponding cumulative distributions $F$ and $G$, our candidate solution to problem \eqref{eq:riemannprONE} for $t$ small enough is given by
\begin{align*}
&F(t,x)=\begin{cases}
		0 \qquad &x\le-1,\\
        \dfrac{x+1}{2t} \qquad &-1\le x\le2t-1,\\
        1 \qquad &x\ge2t-1,
	   \end{cases}\\[1em]
&G(t,x)=0 \qquad \mbox{for}\quad x<1,
\end{align*}
whereas for problem \eqref{eq:riemannprTWO} we have
\begin{align*}
&F(t,x)= 1 \qquad \mbox{for}\quad x\ge 2t-1,\\
&G(t,x)=\begin{cases}
		0 \qquad &0\le x<-2t+1,\\
        1+\dfrac{x-1}{2t} \qquad &-2t+1\le x\le1\\
        1 \qquad &x>1.
	   \end{cases}
\end{align*}
The composition of these two solutions for short times is represented in Figure \ref{fig:figura_dist}.
\begin{figure}[!ht]
    \begin{center}
        \begin{minipage}[c]{.48\textwidth}
            \includegraphics[width=1.1\textwidth]{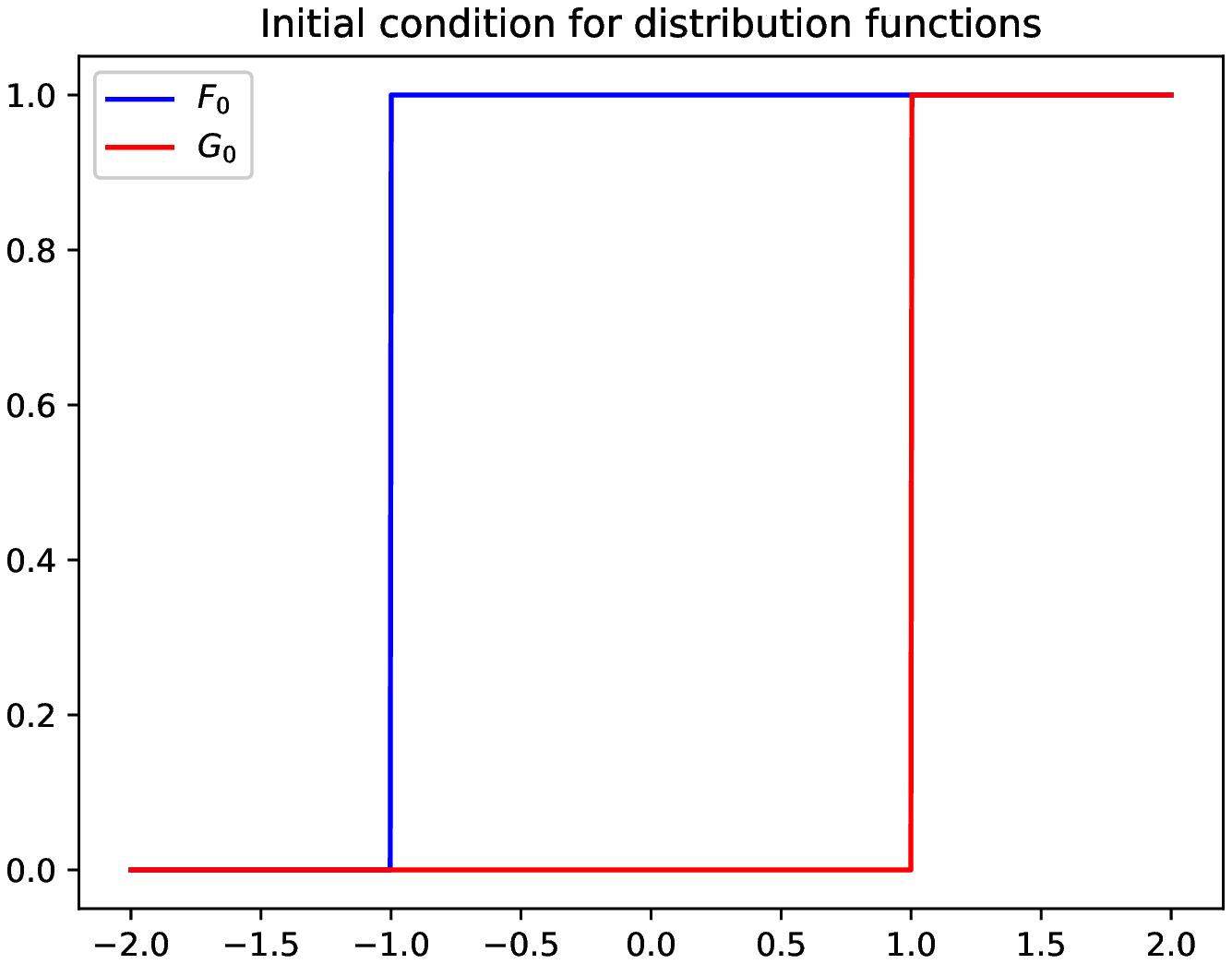}
        \end{minipage}
        \begin{minipage}[c]{.48\textwidth}
            \includegraphics[width=1.1\textwidth]{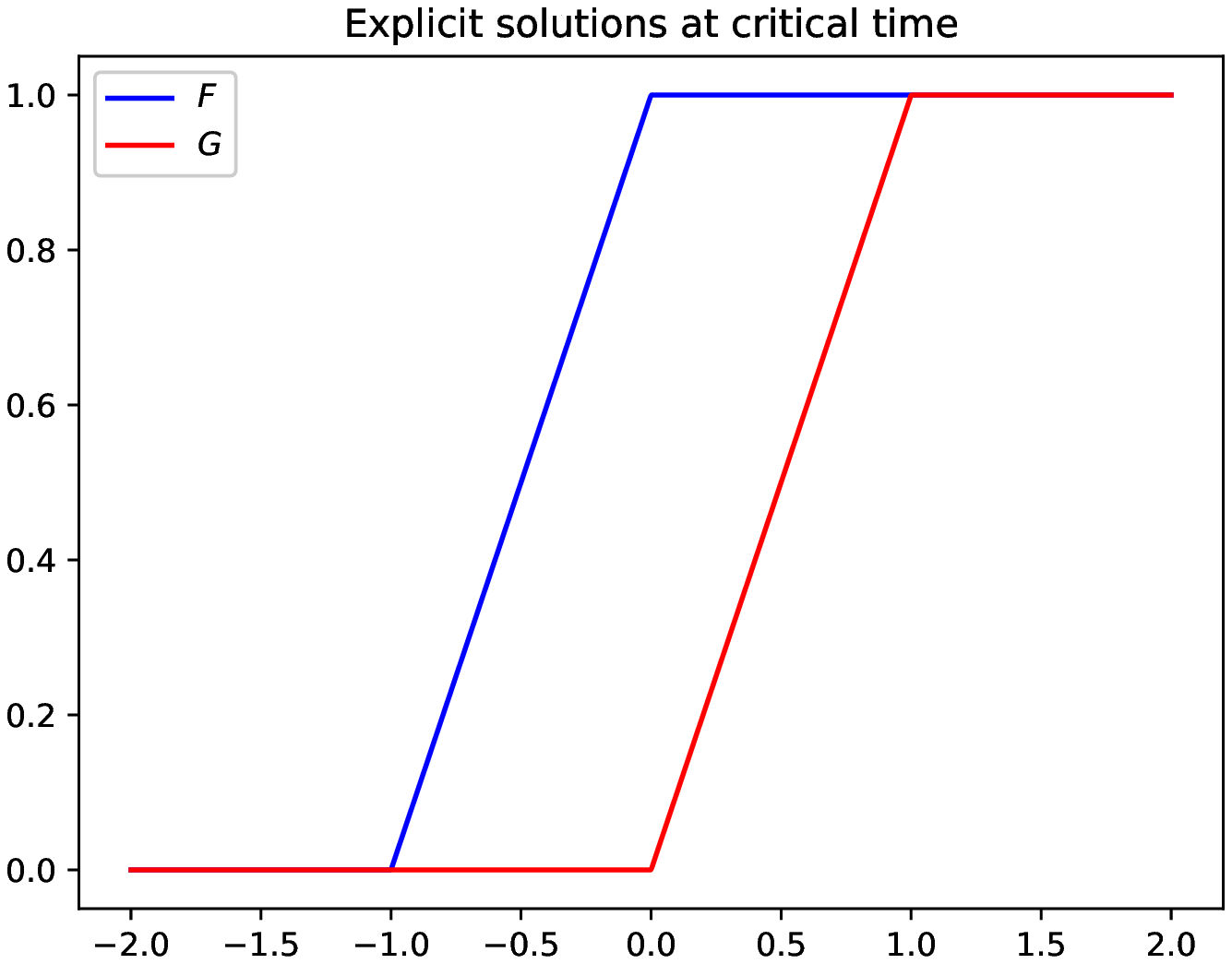}
        \end{minipage}
    \end{center}
    \caption{Initial (left) and exact solution (right) at time $t=0.5$ for the case of two distinct Dirac deltas at the level of distribution functions.}
    \label{fig:figura_dist}
\end{figure}

Let us now consider as initial datum the cumulative distribution functions of
\begin{align*}
  \rho_0=\delta_0,\quad\mbox{and}\quad \eta_0=m \delta_0 + (1-m)\delta_1,
 \end{align*}
as in \eqref{eq:two_over}. As before we have to deal with two different Riemann problems, i.e.
\begin{equation}\label{eq:riemannprONEbis}
\begin{cases}
	\partial_t F +2(F-G) \partial_x F=0\\
    \partial_t G +2(G-F) \partial_x G=0,
\end{cases}
\quad F_0=\begin{cases}
0 \quad x<0\\
1 \quad x\geq 0
\end{cases}, \quad
G_0=\begin{cases}
0 \quad x<0\\
m \quad x\geq 0
          \end{cases}.
\end{equation}
and
\begin{equation}\label{eq:riemannprTWObis}
\begin{cases}
	\partial_t F +2(F-G) \partial_x F=0\\
    \partial_t G +2(G-F) \partial_x G=0,
\end{cases}
\quad F_0=\begin{cases}
1\quad 0<x<1\\
1 \quad x\geq 1
\end{cases}, \quad
G_0=\begin{cases}
			m \quad x<1\\
           1 \quad x\geq 1
          \end{cases}.
\end{equation}
Going back to the results of the previous subsection, the gradient flow solution for the pseudo-inverse system \eqref{eq:system-pseudo-inverse} is given by
\begin{align*}
&X(t,z)=\begin{cases}
		0 \qquad &0\le z<m\\
	    2t(z-m)  &m\le z\le1,
       \end{cases}
&Y(t,z)=\begin{cases}
		0 \qquad &0\le z<m\\
	    2t(z-1)+1  &m\le z\le1,
       \end{cases}
\end{align*}
and then, for $t$ small enough, the candidate solution to problem \eqref{eq:riemannprONEbis} is given by
\begin{align*}
&F(t,x)=\begin{cases}
		0 \qquad &x<0,\\
        m+\dfrac{x}{2t} \qquad &0\le x\le2(1-m)t,\\
        1 \qquad &x>2(1-m)t,
	   \end{cases}\\
&G(t,x)=\begin{cases}
		0 \qquad &x<0,\\
        m \qquad &x\ge0,
	   \end{cases}
\end{align*}
whereas for problem \eqref{eq:riemannprTWObis} we have
\begin{align*}
&F(t,x)= 1 \qquad \mbox{for}\quad x>2(1-m)t,\\
&G(t,x)=\begin{cases}
		m \qquad &0\le x<2(m-1)t+1,\\
        1+\dfrac{x-1}{2t} \qquad &2(m-1)t+1\le x\le1\\
        1 \qquad &x>1,
	   \end{cases}
\end{align*}
see Figure \ref{fig:figura_over}.

\begin{figure}[!ht]
    \begin{center}
        \begin{minipage}[c]{.48\textwidth}
            \includegraphics[width=1.1\textwidth]{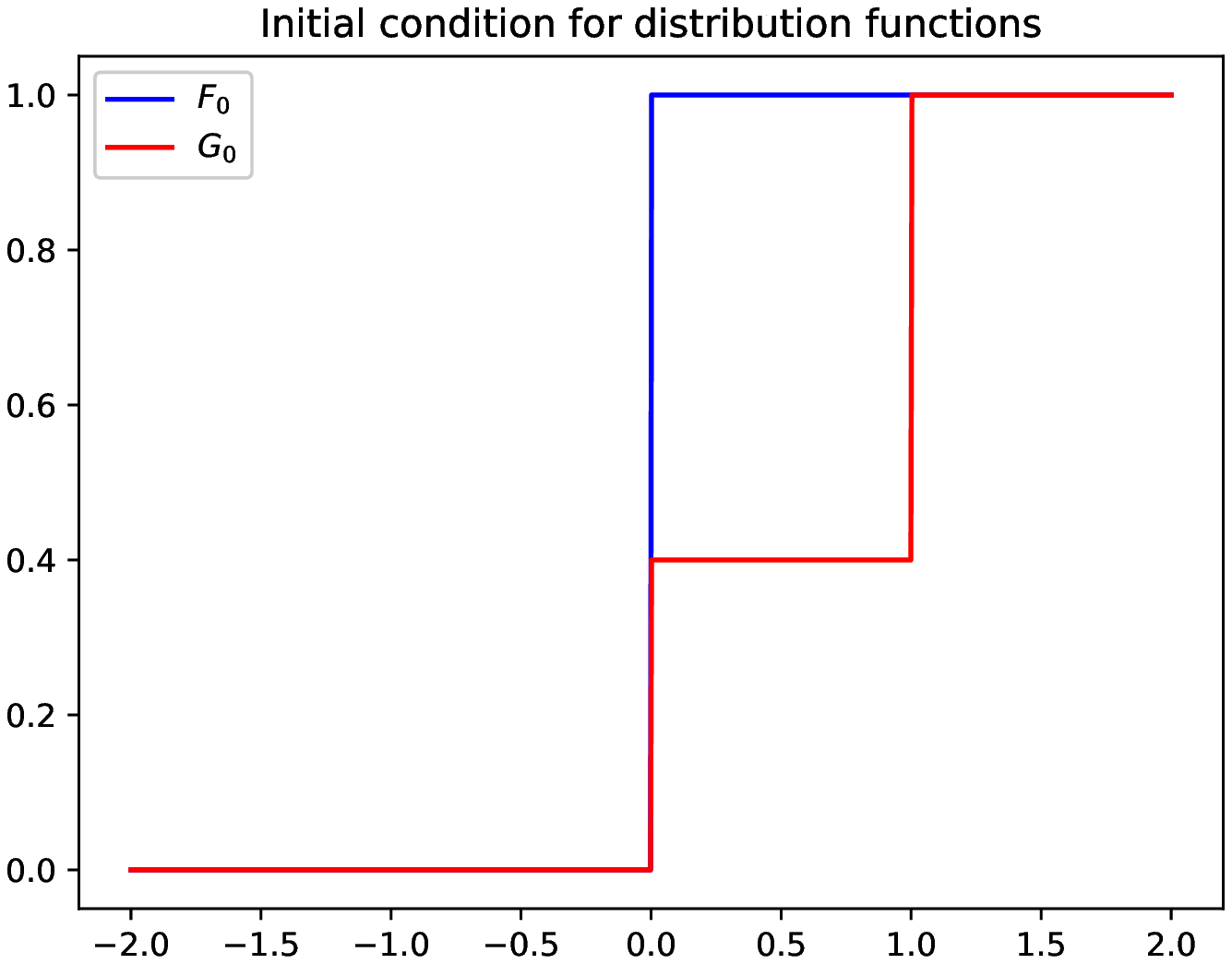}
        \end{minipage}
        \begin{minipage}[c]{.48\textwidth}
            \includegraphics[width=1.1\textwidth]{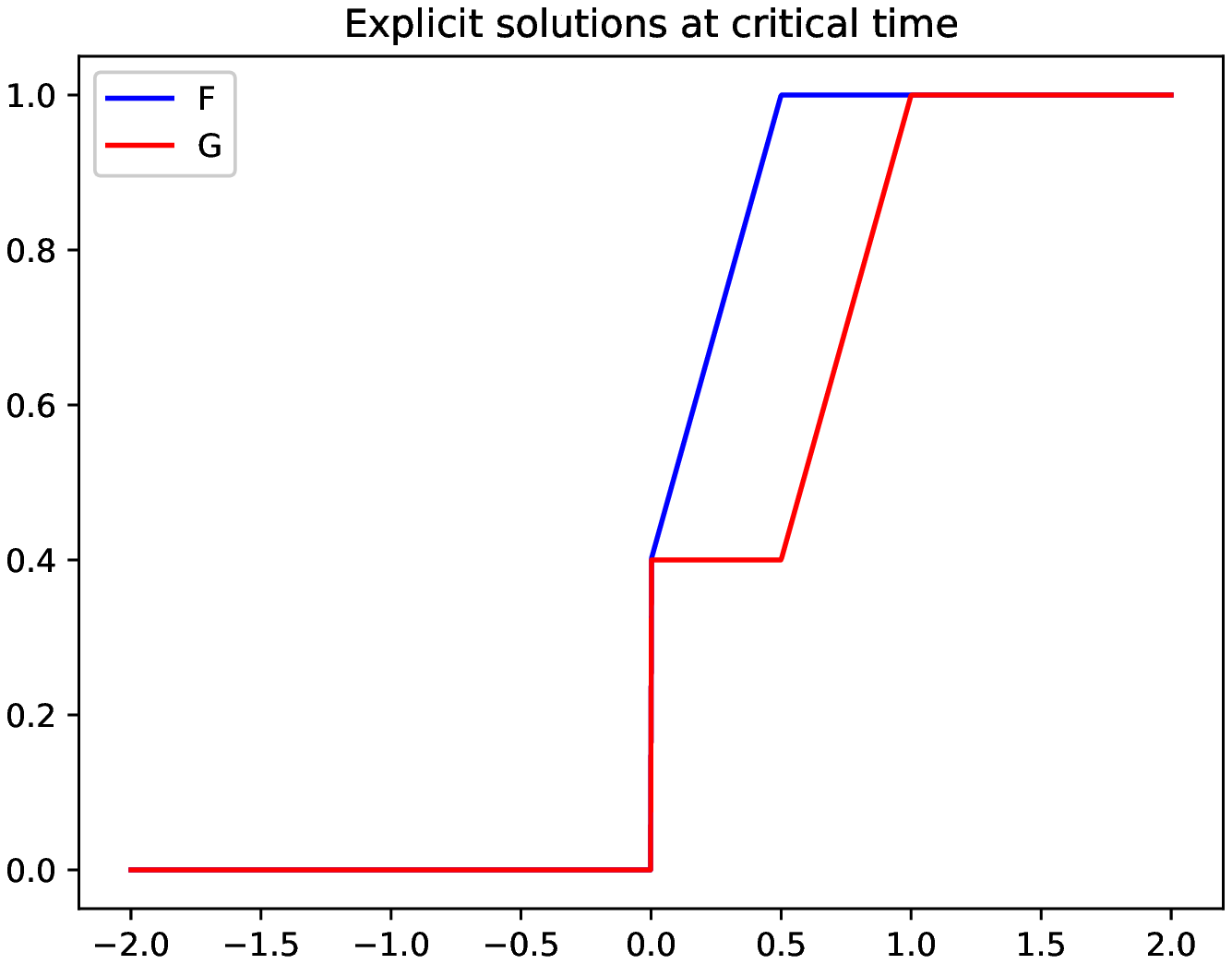}
        \end{minipage}
    \end{center}
    \caption{Initial (left) and exact solution (right) at time $t=1/(4(1-m))
    $ with $m=0.4$ for the case of two partially overlapping deltas at the level of distribution functions.}
    \label{fig:figura_over}
\end{figure}

\section*{Acknowledgements}
JAC was partially supported by the EPSRC grant number  EP/P031587/1. The authors acknowledge discussions with D. Amadori and F. Ancona on the link with hyperbolic systems. MDF, AE, and SF acknowledge support by the EU-funded Erasmus Mundus programme "MathMods - Mathematical models in engineering: theory, methods, and applications" at the University of L’Aquila, by the Italian GNAMPA mini-project "Analisi di modelli matematici della fisica, della biologia e delle scienze sociali", and by the local fund of the University of L'Aquila ‘DP-LAND (Deterministic Particles for Local And Nonlocal Dynamics).

\bibliography{references}

\def\cprime{$'$}
\begin{thebibliography}{10}

\bibitem{AGS}
L.~Ambrosio, N.~Gigli, and G.~Savar{\'e}.
\newblock {\em Gradient flows in metric spaces and in the space of probability
  measures}.
\newblock Lectures in Mathematics ETH Z\"urich. Birkh\"auser Verlag, Basel,
  second edition, 2008.

\bibitem{AS07}
L.~Ambrosio and G.~Savar{\'e}.
\newblock Gradient flows of probability measures.
\newblock {\em Handbook of differential equations: evolutionary equations},
  3:1--136, 2007.

\bibitem{arnold}
A.~Arnold, P.~Markowich, and G.~Toscani.
\newblock On large time asymptotics for drift-diffusion-{P}oisson systems.
\newblock In {\em Proceedings of the Fifth International Workshop on
  Mathematical Aspects of Fluid and Plasma Dynamics (Maui, HI, 1998)},
  volume~29, pages 571--581, 2000.

\bibitem{BCLR2}
D.~Balagu\'e, J.~A. Carrillo, T.~Laurent, and G.~Raoul.
\newblock Dimensionality of local minimizers of the interaction energy.
\newblock {\em Arch. Ration. Mech. Anal.}, 209(3):1055--1088, 2013.

\bibitem{BCP97}
D.~Benedetto, E.~Caglioti, and M.~Pulvirenti.
\newblock A kinetic equation for granular media.
\newblock {\em RAIRO-Mod{\'e}lisation math{\'e}matique et analyse
  num{\'e}rique}, 31(5):615--641, 1997.

\bibitem{BCL}
A.~L. Bertozzi, J.~A. Carrillo, and T.~Laurent.
\newblock Blow-up in multidimensional aggregation equations with mildly
  singular interaction kernels.
\newblock {\em Nonlinearity}, 22(3):683--710, 2009.

\bibitem{BKSUV}
A.~L. Bertozzi, T.~Kolokolnikov, H.~Sun, D.~Uminsky, and J.~von Brecht.
\newblock Ring patterns and their bifurcations in a nonlocal model of
  biological swarms.
\newblock {\em Commun. Math. Sci.}, 13(4):955--985, 2015.

\bibitem{BLL12}
A.~L. Bertozzi, T.~Laurent, and F.~L{\'e}ger.
\newblock Aggregation and spreading via the newtonian potential: the dynamics
  of patch solutions.
\newblock {\em Mathematical Models and Methods in Applied Sciences},
  22(supp01):1140005, 2012.

\bibitem{BLR11}
A.~L. Bertozzi, T.~Laurent, and J.~Rosado.
\newblock Lp theory for the multidimensional aggregation equation.
\newblock {\em Communications on Pure and Applied Mathematics}, 64(1):45--83,
  2011.

\bibitem{Boll-br-loe}
F.~Bolley, Y.~Brenier, and G.~Loeper.
\newblock Contractive metrics for scalar conservation laws.
\newblock {\em J. Hyperbolic Differ. Equ.}, 2(1):91--107, 2005.

\bibitem{B}
G.~A. Bonaschi.
\newblock Gradient flows driven by a non-smooth repulsive interaction
  potential.
\newblock {\em Master’s thesis, University of Pavia, Italy, arXiv:1310.3677},
  2011.

\bibitem{BCDFP}
G.~A. Bonaschi, J.~A. Carrillo, M.~Di~Francesco, and M.~A. Peletier.
\newblock Equivalence of gradient flows and entropy solutions for singular
  nonlocal interaction equations in 1d.
\newblock {\em ESAIM Control, Optimisation and Calculus of Variations},
  21(2):414--441, 2015.

\bibitem{Brenier}
Y.~Brenier.
\newblock L$^2$ formulation of multidimensional scalar conservation laws.
\newblock {\em Arch. Ration. Mech. Anal.}, 193(1):1--19, 2009.

\bibitem{Bressan}
A.~Bressan.
\newblock {\em Hyperbolic systems of conservation laws. The one-dimensional
  Cauchy problem}.
\newblock Oxford Lecture Series in Mathematics and its Applications. Oxford
  University Press, Oxford, 20 edition, 2000.

\bibitem{Brezis}
H.~Br\'{e}zis.
\newblock {\em Op\'{e}rateurs maximaux monotones et semi-groupes de
  contractions dans les espaces de Hilbert}.
\newblock North-Holland Publishing Co., Amsterdam, 1973.

\bibitem{BDFF11}
M.~Burger, M.~{Di Francesco}, and M.~Franek.
\newblock Stationary states of quadratic diffusion equations with long-range
  attraction.
\newblock {\em Communications in Mathematical Sciences}, 11(3):709--738, 2013.

\bibitem{BFH}
M.~Burger, R.~Fetecau, and Y.~Huang.
\newblock Stationary states and asymptotic behavior of aggregation models with
  nonlinear local repulsion.
\newblock {\em SIAM J. Appl. Dyn. Syst.}, 13(1):397--424, 2014.

\bibitem{BDFS}
M.~Burger, M.~Francesco, S.~Fagioli, and A.~Stevens.
\newblock Sorting phenomena in a mathematical model for two mutually
  attracting/repelling species.
\newblock {\em SIAM Journal on Mathematical Analysis}, 50(3):3210--3250, 2018.

\bibitem{CCH1}
V.~Calvez, J.~A. Carrillo, and F.~Hoffmann.
\newblock Equilibria of homogeneous functionals in the fair-competition regime.
\newblock {\em Nonlinear Anal.}, 159:85--128, 2017.

\bibitem{CCH2}
V.~Calvez, J.~A. Carrillo, and F.~Hoffmann.
\newblock The geometry of diffusing and self-attracting particles in a
  one-dimensional fair-competition regime.
\newblock 2186:1--71, 2017.

\bibitem{CHS17}
J.~Carrillo, Y.~Huang, and M.~Schmidtchen.
\newblock Zoology of a nonlocal cross-diffusion model for two species.
\newblock {\em SIAM Journal on Applied Mathematics}, 78(2):1078--1104, 2018.

\bibitem{CCH15}
J.~A. Carrillo, A.~Chertock, and Y.~Huang.
\newblock A finite-volume method for nonlinear nonlocal equations with a
  gradient flow structure.
\newblock {\em Communications in Computational Physics}, 17(01):233--258, 2015.

\bibitem{CCY}
J.~A. Carrillo, K.~Craig, and Y.~Yao.
\newblock Aggregation-diffusion equations: dynamics, asymptotics, and singular
  limits.
\newblock {\em arXiv preprint arXiv:1810.03634}, 2018.

\bibitem{CDM16}
J.~A. Carrillo, M.~G. Delgadino, and A.~Mellet.
\newblock Regularity of local minimizers of the interaction energy via obstacle
  problems.
\newblock {\em Comm. Math. Phys.}, 343(3):747--781, 2016.

\bibitem{CDFFLS}
J.~A. Carrillo, M.~Di~Francesco, A.~Figalli, T.~Laurent, and D.~Slepcev.
\newblock Global-in-time weak measure solutions and finite-time aggregation for
  nonlocal interaction equations.
\newblock {\em Duke Math. J.}, 156(2):229--271, 2011.

\bibitem{CFP12}
J.~A. Carrillo, L.~C. Ferreira, and J.~C. Precioso.
\newblock A mass-transportation approach to a one dimensional fluid mechanics
  model with nonlocal velocity.
\newblock {\em Advances in Mathematics}, 231(1):306--327, 2012.

\bibitem{CFS18}
J.~A. {Carrillo}, F.~{Filbet}, and M.~{Schmidtchen}.
\newblock {Convergence of a Finite Volume Scheme for a System of Interacting
  Species with Cross-Diffusion}.
\newblock {\em ArXiv e-prints}, Apr. 2018.

\bibitem{CHM14}
J.~A. Carrillo, Y.~Huang, and S.~Martin.
\newblock Explicit flock solutions for {Q}uasi-{M}orse potentials.
\newblock {\em European J. Appl. Math.}, 25(5):553--578, 2014.

\bibitem{CMP13}
J.~A. Carrillo, S.~Martin, and V.~Panferov.
\newblock A new interaction potential for swarming models.
\newblock {\em Phys. D}, 260:112--126, 2013.

\bibitem{CT}
J.~A. Carrillo and G.~Toscani.
\newblock Wasserstein metric and large--time asymptotics of nonlinear diffusion
  equations.
\newblock {\em New Trends in Mathematical Physics, (In Honour of the Salvatore
  Rionero 70th Birthday)}, pages 234--244, 2005.

\bibitem{HarCaz}
T.~Cazenave and A.~Haraux.
\newblock {\em An introduction to semilinear evolution equations}.
\newblock Oxford Lecture Series in Mathematics and its Applications, 13. Oxford
  University Press, New York, 1998.

\bibitem{Dafermos}
C.~M. Dafermos.
\newblock {\em Hyperbolic conservation laws in continuum physics. Fourth
  edition}.
\newblock Grundlehren der Mathematischen Wissenschaften [Fundamental Principles
  of Mathematical Sciences]. Springer-Verlag, Berlin, 325 edition, 2016.

\bibitem{DS}
S.~Daneri and G.~Savar\'{e}.
\newblock Eulerian calculus for the displacement convexity in the wasserstein
  distance.
\newblock {\em SIAM J. Math. Anal.}, 40(3):1104--–1122, 2008.

\bibitem{DeGiorgi}
E.~De~Giorgi.
\newblock New problems on minimizing movements.
\newblock {\em Boundary Value Problems for PDE and Applications, C. Baiocchi
  and J. L. Lions eds., Masson}, pages 81–--98, 1993.

\bibitem{DiFranEspFag}
M.~Di~Francesco, A.~Esposito, and S.~Fagioli.
\newblock Nonlinear degenerate cross-diffusion systems with nonlocal
  interaction.
\newblock {\em Nonlinear Analysis}, 169:94--117, 2018.

\bibitem{DFF}
M.~Di~Francesco and S.~Fagioli.
\newblock Measure solutions for nonlocal interaction pdes with two species.
\newblock {\em Nonlinearity}, 26:2777--2808, 2013.

\bibitem{DFM}
M.~Di~Francesco and D.~Matthes.
\newblock Curves of steepest descent are entropy solutions for a class of
  degenerate convection-diffusion equations.
\newblock {\em Calc. Var. Partial Differential Equations}, 50(1--2):199--230,
  2014.

\bibitem{OCBC06}
M.~R. D'Orsogna, Y.-L. Chuang, A.~L. Bertozzi, and L.~S. Chayes.
\newblock Self-propelled particles with soft-core interactions: patterns,
  stability, and collapse.
\newblock {\em Physical review letters}, 96(10):104302, 2006.

\bibitem{evans}
L.~C. Evans.
\newblock {\em {Partial Differential Equations}}, volume~19 of {\em Graduate
  Studies in Mathematics}.
\newblock {American Mathematical Society}, Providence, Rhode Island, 1998.

\bibitem{EK}
J.~H.~M. Evers and T.~Kolokolnikov.
\newblock Metastable states for an aggregation model with noise.
\newblock {\em SIAM J. Appl. Dyn. Syst.}, 15(4):2213--2226, 2016.

\bibitem{FH}
R.~C. Fetecau and Y.~Huang.
\newblock Equilibria of biological aggregations with nonlocal
  repulsive-attractive interactions.
\newblock {\em Phys. D}, 260:49--64, 2013.

\bibitem{FHK}
R.~C. Fetecau, Y.~Huang, and T.~Kolokolnikov.
\newblock Swarm dynamics and equilibria for a nonlocal aggregation model.
\newblock {\em Nonlinearity}, 24(10):2681--2716, 2011.

\bibitem{JKO}
R.~Jordan, D.~Kinderlehrer, and F.~Otto.
\newblock The variational formulation of the fokker--planck equation.
\newblock {\em SIAM journal on mathematical analysis}, 29(1):1--17, 1998.

\bibitem{KSUB}
T.~Kolokolnikov, H.~Sun, D.~Uminsky, and A.~L. Bertozzi.
\newblock Stability of ring patterns arising from two-dimensional particle
  interactions.
\newblock {\em Phys. Rev. E}, 84:015203, Jul 2011.

\bibitem{kru}
S.~Kru{\v{z}}hkov.
\newblock First order quasilinear equations with several independent variables.
\newblock {\em Mat. Sb. (N.S.)}, 81 (123):228--255, 1970.

\bibitem{LSU}
O.~A. Lady{\v{z}}enskaja, V.~A. Solonnikov, and N.~N. Ural'ceva.
\newblock {\em Linear and quasilinear equations of parabolic type}.
\newblock Translated from the Russian by S. Smith. Translations of Mathematical
  Monographs, Vol. 23. American Mathematical Society, Providence, R.I., 1967.

\bibitem{lt}
H.~Li and G.~Toscani.
\newblock Long-time asymptotics of kinetic models of granular flows.
\newblock {\em Arch. Ration. Mech. Anal.}, 172(3):407–--428, 2004.

\bibitem{liu}
T.-P. Liu and M.~Pierre.
\newblock Source solutions and asymptotic behaviour in conservations laws.
\newblock {\em J. Differential Equations}, 51:419--441, 1984.

\bibitem{MMCS}
D.~Matthes, R.~McCann, and G.~Savar\'{e}.
\newblock A family of fourth order equations of gradient flow type.
\newblock {\em Comm. P.D.E.}, 34(11):1352--1397, 2009.

\bibitem{McC97}
R.~J. McCann.
\newblock A convexity principle for interacting gases.
\newblock {\em Advances in mathematics}, 128(1):153--179, 1997.

\bibitem{MEK99}
A.~Mogilner and L.~Edelstein-Keshet.
\newblock A non-local model for a swarm.
\newblock {\em Journal of Mathematical Biology}, 38(6):534--570, 1999.

\bibitem{otto}
F.~Otto.
\newblock The geometry of dissipative evolution equations: the porous medium
  equation.
\newblock {\em Comm. Partial Differential Equations}, 26(1--2):101--174, 2001.

\bibitem{RR}
S.~T. Rachev and L.~R\"{u}schendorf.
\newblock {\em Mass Transportation Problems}, volume~I of {\em Probability and
  Its Applications}.
\newblock Springer, New York, 1998.

\bibitem{S}
F.~Santambrogio.
\newblock {\em Optimal Transport for Applied Mathematicians}, volume~86 of {\em
  Progress in Nonlinear Differential Equations and Their Applications}.
\newblock Birkh\"auser Verlag, Basel, 2015.

\bibitem{TBL06}
C.~M. Topaz, A.~L. Bertozzi, and M.~A. Lewis.
\newblock A nonlocal continuum model for biological aggregation.
\newblock {\em Bull. Math. Biol.}, 68(7):1601--1623, 2006.

\bibitem{Tos00}
G.~Toscani.
\newblock One-dimensional kinetic models of granular flows.
\newblock {\em ESAIM: Mod{\'e}lisation Math{\'e}matique et Analyse
  Num{\'e}rique}, 34(6):1277--1291, 2000.

\bibitem{V1}
C.~Villani.
\newblock {\em Topics in optimal transportation}, volume~58 of {\em Graduate
  Studies in Mathematics}.
\newblock American Mathematical Society, Providence, RI, 2003.

\bibitem{V2}
C.~Villani.
\newblock {\em Optimal transport : old and new}.
\newblock Grundlehren der mathematischen Wissenschaften. Springer, Berlin,
  2009.

\end{thebibliography}
\bibliographystyle{abbrv}

\end{document}